\documentclass[12pt]{amsart}
\usepackage{amssymb,times,mathabx, bbm,bm,color,amsmath}

\parskip=\medskipamount

\newtheorem*{Thm*}{Theorem}
\newtheorem{Thm}{Theorem}
\newtheorem{Cor}[Thm]{Corollary}
\newtheorem{Prop}[Thm]{Proposition}
\newtheorem{Lemma}[Thm]{Lemma}

\theoremstyle{definition}
\newtheorem{Defn}[Thm]{Definition}

\newtheorem{Remark}[Thm]{Remark}
\newtheorem{Example}[Thm]{Example}

\newcommand{\real}{\mathbb{R}}
\newcommand{\comp}{\mathbb{C}}

\newcommand{\mf}[1]{\mathbb{#1}}
\newcommand{\mc}[1]{\mathcal{#1}}
\newcommand{\mb}[1]{\mathbf{#1}}

\DeclareMathOperator{\supp}{\mathrm{supp}}

\DeclareMathOperator{\Imm}{\mathit{Im}}

\newcommand{\abs}[1]{\left\vert#1\right\vert}

\newcommand{\set}[1]{\left\{#1\right\}}

\renewcommand{\phi}{\varphi}

\newcommand{\eps}{\varepsilon}

\newcommand{\msubord}{\ \frame{$\ \rightthreetimes$}\; }

\newcommand{\calt}{\circlearrowright}

\newcommand{\utimes}{\kern0.05em\buildrel{\times}\over{\rule{0em}{0.004em}} \kern-0.9em\cup \kern0.2em}
\newcommand{\putimes}{\mathop{\kern0.05em\buildrel{\times}\over{\rule{0em}{0.0em}} \kern-0.9em\cup \kern0em}}
\newcommand{\hutimes}{\mathop{\kern0.02em\buildrel{\times}\over{\rule{0em}{0.0em}} \kern-0.48em\cup \kern0em}}
\newcommand{\sutimes}{\mathrel{\kern0em \buildrel{\mathsf{x}}\over{\rule{0em}{0.0em}} \kern-0.35em\cup \kern-0.0em}}

\allowdisplaybreaks[1]

\title{The exponential map in non-commutative probability}
\author{Michael Anshelevich}
\thanks{The first author was supported in part by NSF grant DMS-1160849.}
\address{Department of Mathematics, Texas A\&M University, College Station, TX 77843-3368}
\email{manshel@math.tamu.edu}
\author{Octavio Arizmendi}
\address{Centro de Investigaci{\'o}n en Matem{\'a}ticas. Guanajuato, Mexico}
\email{octavius@cimat.mx}
\thanks{The second author was supported by CONACYT Grant 222668.}
\subjclass[2010]{Primary 46L54; Secondary 46L53}
\date{\today}

\begin{document}

\begin{abstract}
The wrapping transformation $W$ is a homomorphism from the semigroup of probability measures on the real line, with the convolution operation, to the semigroup of probability measures on the circle, with the multiplicative convolution operation. We show that on a large class $\mc{L}$ of measures, $W$ also transforms the three non-commutative convolutions---free, Boolean, and monotone---to their multiplicative counterparts. Moreover, the restriction of $W$ to $\mc{L}$ preserves various qualitative properties of measures and triangular arrays. We use these facts to give short proofs of numerous known, and new, results about multiplicative convolutions.
\end{abstract}

%\begin{center}
%\textsl{Preliminary version}
%\end{center}

\maketitle

\section{Introduction}

In probability theory, the study of sums of random variables is closely related to the study of the convolution operation $\ast$ on $\mc{P}(\mf{R})$, the probability measures on the real line. Similarly, the study of products of random variables is closely related to the study of multiplicative convolution operations $\circledast$ on probability measures on the positive real line $\mf{R}_+$ or on the circle $\mf{T}$,
\[
\int_{\mf{R}_+ \text{ or } \mf{T}} f(z) \,d(\nu_1 \circledast \nu_2)(z) = \int_{\mf{R}_+ \text{ or } \mf{T}} \int_{\mf{R}_+ \text{ or } \mf{T}} f(z w) \,d\nu_1(z) \,d\nu_1(w).
\]
The measures on $\mf{R}_+$ arise from those on $\mf{R}$ by a simple change of variable $d\nu(e^x) = d\mu(x)$, so that
\begin{equation}
\label{Exp-line}
\int_{\mf{R}} f(e^x) \,d\mu(x) = \int_{\mf{R}_+} f(y) \,d\nu(y).
\end{equation}
Its counterpart for the circle is what we will call the wrapping map.

\begin{Defn}
\label{Defn:Wrapping}
The wrapping map $W : \mc{P}(\mf{R}) \rightarrow \mc{P}(\mf{T})$ is
\begin{equation}
\label{Wrapping}
d(W(\mu))(e^{-ix}) = \sum_{n \in \mf{Z}} d\mu(x + 2 \pi n).
\end{equation}
Note that the map is clearly well defined, and that the measure $\mu$ gets wrapped clockwise, to fit better with the later results in the article. It has the property that
\begin{equation}
\label{Exp:circle}
\int_{\mf{R}} f(e^{-ix}) \,d\mu(x) = \int_{\mf{T}} f(\zeta) \,d(W(\mu))(\zeta).
\end{equation}
\end{Defn}
It is easy to see from \eqref{Exp-line} and \eqref{Exp:circle} that in both cases, these maps transform the additive convolution $\ast$ into the multiplicative convolution $\circledast$. So the study of products of random variables can largely be reduced to the study of sums of such variables, by taking the logarithm. See e.g.\ Chapters I.10 and XIX.5 of \cite{Feller-volume2}, \cite{Schatte-Prob-mod-2pi} and other sources, and \cite{Grenander} for a more general framework.

This is not the case in probability theories dealing with non-commuting variables, for which logarithm does \emph{not} linearize the product. Free probability is the most important such theory, but Boolean and monotone probability theories have also been studied. In these three cases, one can define additive convolutions of measures on the real line, denoted by $\boxplus$, $\uplus$, and $\rhd$, respectively, and corresponding to the addition of (appropriately) independent self-adjoint operators. Multiplicative convolutions, denoted by $\boxtimes$, $\putimes$, and $\circlearrowleft$, respectively, can be again defined for measures on $\mf{R}_+$, corresponding to multiplication of (appropriately) independent positive operators, and for measures on $\mf{T}$, corresponding to the multiplication of unitary operators. In all three cases, multiplicative theories exhibit strong parallels with the additive theories, through various L{\'e}vy-Khinchin-type formulas, limit theorems etc. However, it was already noted in \cite{BV93} that in the case of the positive real line, the behavior of $\boxtimes$ is different from that of $\boxplus$, and so cannot be reduced to it. While such results were not known for the circle, the theorems in that case were proved separately, see \cite{Belinschi-Atoms-mult,Bel-Ber-Partially-mult,Bel-Ber-Hincin,Bercovici-Wang-Multiplicative,Chist-Gotze-Arithmetic,Ariz-Hasebe-Semigroups,Zhong-Free-BM} for a partial list. Only a posteriori did they turn out to be similar to those on the real line.
%Check for more results to re-prove.

Still, as indicated for example in Section 4 of \cite{Ariz-Hasebe-Semigroups}, a hope for some replacement for a logarithm lingers, and several successful definitions have been given, in the algebraic setting (for general functionals rather than measures). In \cite{Nica-Mastnak}, the authors introduced a transformation based on character theory, which linearizes the multiplicative free convolution $\boxtimes$, in certain cases even in a multivariate setting. In \cite{Friedrich-McKay-Witt,Friedrich-McKay-Almost} the authors constructed another relation between additive and multiplicative instances of the free theory, again on the level of power series. The third approach due to C{\'e}bron \cite{Cebron-Multiplicative}, which is most closely related to ours, is described in Remark~\ref{Remark:Cebron}.

In this paper, we show that the wrapping map can, after all, be used to relate the free convolutions on $\mf{R}$ and $\mf{T}$. We do this by re-casting it as a very simple (exponential) transformation between analytic-function transforms. It is easy to see (Example~\ref{Example:Counter}) that $W$ is not a homomorphism between $(\mc{P}(\mf{R}), \boxplus)$ and $(\mc{P}(\mf{T}), \boxtimes)$. However, we show that it is a homomorphism from a certain class of measures $\mc{L}$, defined below, to $\mc{ID}^{\sutimes}_\ast$, the space of all probability measures on $\mf{T}$ other then the Lebesgue measure, infinitely divisible with respect to Boolean convolution. Our approach has two new features. First, it is analytic rather than algebraic. In fact, measures in $\mc{L}$ do not have finite moments, and all lie in the domain of attraction of the Cauchy law. Second, the same transformation $W$ is a homomorphism not just between free convolutions, but between Boolean and monotone additive and multiplicative convolutions as well.  In particular, we show that $\mc{ID}^{\sutimes}_\ast$, taken with any of the three multiplicative convolutions, is isomorphic (modulo a simple equivalence relation) to $\mc{L}$ taken with the corresponding additive convolution. This is of interest since the three non-commutative probability theories are actually related, and our techniques allow us to handle multiplicative versions of objects involving more than one convolution, such as the Belinschi-Nica transformations or subordination distributions. In addition, when restricted to $\mc{L}$, the wrapping map intertwines the additive Bercovici-Pata bijections (including the classical one) with their multiplicative counterparts.

The wrapping map preserves numerous properties of the measure, such as having finitely many atoms, and maps infinitesimal triangular arrays to infinitesimal triangular arrays (recall that roughly, measures in an infinitesimal array on $\mf{R}$, respectively $\mf{T}$, concentrate around $0$, respectively, $1$, and see Lemma~\ref{Lemma:Infinitesimal} for a precise definition). Clearly the converses fail: under application of $W$, several atoms may coalesce, and a pre-image of an infinitesimal array is only guaranteed to concentrate in the neighborhoods of multiples of $2 \pi$. Nevertheless, when restricted to the map $W: \mc{L} \rightarrow \mc{ID}^{\sutimes}_\ast$, the unwrapping map preserves such properties as well. When so restricted, $W$ and $W^{-1}$
\begin{itemize}
\item
Preserve the number and weights of the atoms.
\item
Preserve the property of being absolutely continuous with a strictly positive density.
\item
Map infinitesimal triangular arrays to infinitesimal triangular arrays.
\end{itemize}
This allows us to carry over results about qualitative properties of free convolution and limit theorems from the additive to the multiplicative case.

The paper is organized as follows. Section~\ref{Section:Background} contains the background material. In Section~\ref{Section:Class}, we describe and study the class $\mc{L}$. In particular, we provide the first examples of measures with connected support none of whose free convolution powers are unimodal. In Section~\ref{Section:Transformation}, we give an alternative definition for the wrapping map $W$, restricted to $\mc{L}$, based on complex-analytic transforms, and study its properties. In Section~\ref{Section:Applications}, we list various applications. New results include various properties of the multiplicative Belinschi-Nica transformations and the divisibility indicator, the multiplicative analog of the inviscid Burgers' equation and the general subordination evolution, several identities, and multiplicative versions of various results of the second author. In addition, we provide numerous simpler proofs of known multiplicative results (typically restricted to $\mc{ID}^{\sutimes}_\ast$), including limit theorems and qualitative properties of free convolution.

\textbf{Acknowledgements.} We would like to thank Takahiro Hasebe for useful discussions, Guillaume C{\'e}bron for reminding us about his work \cite{Cebron-Multiplicative}, and Yoshimichi Ueda for pointing out the connection with the Loewner's equation.  We are grateful to Uwe Franz and Takahiro Hasebe for allowing us to include Lemma \ref{uni2}. Finally, we would like to thank the referees for helpful comments.

\section{Background}
\label{Section:Background}

\subsection{Transforms for additive convolutions}

For $\mu$ a probability measure on $\mf{R}$, we define the Cauchy transform
\[
G_\mu(z) = \int_{\mf{R}} \frac{1}{z - x} \,d\mu(x).
\]
and the $F$-transform $F_\mu : \mf{C}^+ \rightarrow \mf{C}^+$,
\[
F(z) = \frac{1}{G(z)}.
\]
$\mu$ can be recovered from these transforms by taking boundary values
\[
d\mu(x)
= - \frac{1}{\pi} \lim_{y \downarrow 0} \Imm G_\mu(x + i y) \,dx
= - \frac{1}{\pi} \lim_{y \downarrow 0} \Imm \frac{1}{F_\mu(x + i y)} \,dx.
\]
$F$ is invertible on a Stolz angle, and the Voiculescu transform of $\mu$ is
\[
\phi_\mu(z) = F_\mu^{-1}(z) - z.
\]
Using these transforms, given  probability measures $\mu$ and $\nu$ on $\mf{R}$, we may define their free, Boolean and monotone additive convolution operations (in symbols $\mu \boxplus \nu$, $\mu \uplus \nu$ and $\mu \rhd \nu$) via
\begin{equation*}
\phi_{\mu \boxplus \nu}(z) = \phi_{\mu}(z) + \phi_{\nu}(z), \quad (F_{\mu \uplus \nu}(z) - z) = (F_\mu(z) - z) + (F_\nu(z) - z), \quad F_{\mu \rhd \nu}(z) = F_\mu(F_\nu(z)).
\end{equation*}
We may also define additive convolution powers via
\[
\phi_{\mu^{\boxplus t}}(z) = t \phi_\mu(z), \quad (F_{\mu^{\uplus t}}(z) - z) = t (F_\mu(z) - z).
\]
$\mu^{\boxplus t}$ is defined at least for $t \geq 1$, and $\mu^{\uplus t}$ for all $t \geq 0$.

The Bercovici-Pata bijections between the classes of distributions infinitely divisible with respect to each of the four additive convolution operations are the bijections between each of $\mc{ID}^\ast$, $\mc{ID}^\boxplus$, $\mc{ID}^\uplus$, and $\mc{ID}^\rhd$, and pairs $(\alpha, \tau)$, for $\alpha \in \mf{R}$ and $\tau$ a finite Borel measure on $\mf{R}$. The bijections are given by
\begin{equation}
\label{LH-add}
\mc{F}_{\nu_\ast^{\alpha, \tau}}(t) = \exp \left( i \alpha t + \int_{\mf{R}} (e^{i x t} - 1 - \frac{i x t}{1 + x^2}) \frac{x^2 + 1}{x^2} \,d\tau(x) \right), \quad t \in \mf{R},
\end{equation}
\[
F_{\nu_{\uplus}^{\alpha, \tau}}(z) =  z  - \alpha + \int_{\mf{R}} \frac{1 + x z}{x-z} \,d\tau(x), \quad z \in \mf{C},
\]
\[
\phi_{\nu_{\boxplus}^{\alpha, \tau}}(z) =  \alpha + \int_{\mf{R}} \frac{1 + x z}{z-x} \,d\tau(x), \quad z \in \mf{C},
\]
\begin{equation}
\label{LH-Monotone}
\Phi(z) =  - \alpha + \int_{\mf{R}} \frac{1 + x z}{x-z} \,d\tau(x), \quad \frac{\partial F_t(z)}{\partial t} = \Phi(F_t(z)), \quad F_{\nu_\rhd^{\alpha, \tau}} = F_1.
\end{equation}

The additive subordination distribution $\mu \boxright \nu$ is defined by the analytic continuation of
\[
F_{\mu \boxright \nu} = F_\nu^{-1} \circ F_{\mu \boxplus \nu}.
\]
The Belinschi-Nica transformations \cite{Belinschi-Nica-B_t} are maps $\mf{B}_t: \mc{P}(\mf{R}) \rightarrow \mc{P}(\mf{R})$, $t \geq 0$, defined by
\[
\mf{B}_t(\mu) = \left( \mu^{\boxplus (1+t)} \right)^{\uplus \frac{1}{1+t}}.
\]

\subsection{Transforms for multiplicative convolutions}

For $\mu$ a probability measure on the unit circle $\mf{T}$, we define the transforms
\[
\psi_\mu(z) = \int_{\mf{T}} \frac{z \zeta}{1 - z \zeta} \,d\mu(\zeta), \quad \eta_\mu(z) = \frac{\psi_\mu(z)}{1 + \psi_{\mu}(z)}.
\]
Note that $\eta_\mu : \mf{D} \rightarrow \mf{D}$ and $\eta_\mu(0) = 0$, so $\abs{\eta_\mu(z)} \leq \abs{z}$ for $z \in \mf{D}$, which implies that $\eta_\mu(z)/z$  also maps $\mf{D}$ to itself. %Note also that formally, $\eta_\mu(z) = 1 - z F_\mu(1/z)$.

The measure can be recovered from the transforms by taking boundary values
\[
d\mu(e^{ix})
= \frac{1}{2\pi} \lim_{r \uparrow 1} (1 + 2 \Re \psi_\mu(r e^{-ix})) \,dx
= \frac{1}{2\pi} \lim_{r \uparrow 1} \Re \frac{1 + \eta_\mu(r e^{-ix})}{1 - \eta_\mu(r e^{-ix})} \,dx.
\]
Throughout the paper, we will restrict our analysis to the measures with $\eta'(0) \neq 0$. In this case we may also define
\[
\Sigma_{\mu}(z) = \frac{\eta_{\mu}^{-1}(z)}{z}
\]
on a neighborhood of zero. Using these transforms, given  probability measures $\mu$ and $\nu$ on $\mf{T}$, we may define their free, Boolean and monotone multiplicative convolution operations (in symbols $\mu \boxtimes \nu$, $\mu \utimes \nu$ and $\mu \circlearrowright \nu$) via
\begin{equation}
\label{Convolutions}
\Sigma_{\mu \boxtimes \nu}(z) = \Sigma_{\mu}(z) \Sigma_{\nu}(z), \quad \frac{\eta_{\mu \sutimes \nu}(z)}{z} = \frac{\eta_{\mu}(z)}{z} \frac{\eta_{\nu}(z)}{z}, \quad \eta_{\mu \circlearrowright \nu}(z) = \eta_{\mu} \circ \eta_{\nu}(z).
\end{equation}
We may also define multiplicative convolution powers (see \cite{Bel-Ber-Partially-mult,Franz-Boolean-circle}), via
\[
\Sigma_{\nu^{\boxtimes t}}(z) = \left(\Sigma_\nu(z)\right)^t, \quad \frac{\eta_{\nu^{\sutimes t}}(z)}{z} = \left(\frac{\eta_\nu(z)}{z}\right)^t.
\]
Note that these transforms are in general well-defined only up to a factor of $e^{2 \pi i n t}$, and so the convolution powers are multi-valued.

We denote by $\mc{ID}^{\sutimes}$ the measures infinitely divisible with respect to $\utimes$, and
\[
\mc{ID}^{\sutimes}_\ast = \mc{ID}^{\sutimes} \setminus \set{\text{Lebesgue measure}}.
\]
Recall from \cite{Franz-Boolean-circle} that
\begin{equation}
\label{Multi-Boolean}
\eta_{\mc{ID}^{\sutimes}_\ast} = \set{\eta_\mu \ |\ \mu \in \mc{ID}^{\sutimes}_\ast}
= \set{\eta: \mf{D} \rightarrow \mf{D} \text{ analytic} \ |\ \eta'(0) \neq 0, \text{ and } \eta(z) = 0 \Leftrightarrow z = 0}
\end{equation}
(for Lebesgue measure on the circle, $\eta(z) = 0$). $\mc{ID}^{\boxtimes}_\ast$ and $\mc{ID}^{\circlearrowright}_\ast$ are measures infinitely divisible with respect to $\boxtimes$, respectively, $\circlearrowright$, again excluding the Lebesgue measures. According to \cite{Wang-Boolean}, there are (Bercovici-Pata) bijections between the classes $\mc{ID}^{\boxtimes}_\ast$ and $\mc{ID}^{\sutimes}_\ast$, and pairs $(\gamma, \sigma)$, for $\gamma \in \mf{T}$ and $\sigma$ a finite Borel measure on $\mf{T}$. The bijections are given by
\begin{equation}
\label{Sigma-transform}
\Sigma_{\nu_\boxtimes^{\gamma, \sigma}}(z) = \gamma \exp \left( \int_{\mf{T}} \frac{1 + \zeta z}{1 - \zeta z} \,d\sigma(\zeta) \right), \quad z \in \mf{D},
\end{equation}
\[
\eta_{\nu_{\sutimes}^{\gamma, \sigma}}(z) = \gamma z \exp \left( - \int_{\mf{T}} \frac{1 + \zeta z}{1 - \zeta z} \,d\sigma(\zeta) \right), \quad z \in \mf{D}.
\]
To each such pair also corresponds an element of $\mc{ID}^{\circledast}_\ast$, via
\begin{equation}
\label{LH-Mult}
\mc{F}_{\nu_\circledast^{\gamma, \sigma}}(p) = \gamma^p \exp \left( \int_{\mf{T}} \frac{\zeta^p - 1 - i p \Imm \zeta}{1 - \Re \zeta} \,d\sigma(\zeta) \right), \quad p \in \mf{Z},
\end{equation}
but this correspondence is neither injective nor onto $\mc{ID}^{\circledast}_\ast$. In addition \cite{Bercovici-Multiplicative-monotonic,Ans-Williams-Chernoff}, the elements of $\mc{ID}^{\circlearrowright}_\ast$ other than delta-measures are in a bijection with pairs $(\beta, \sigma)$ for $\beta \in \mf{R}$ and $\sigma$ a non-zero, finite measure on $\mf{T}$, through
\begin{equation}
\label{Monotone-mult-generator}
\begin{split}
& A^{\beta, \sigma}(z) = z \left(- i \beta - \int_{\mf{T}} \frac{1 + \zeta z}{1 - \zeta z} \,d\sigma(\zeta) \right), \\
& \frac{d \eta_{\mu_t}(z)}{d t} = A^{\beta, \sigma}(\eta_{\mu_t}(z)), \quad \eta_{\mu_0}(z) = z, \quad \nu_\circlearrowright^{\beta, \sigma} = \mu_1
\end{split}
\end{equation}
(with the delta-measures corresponding to $A^{\beta, 0}$, $0 \leq \beta < 2 \pi$). The multiplicative subordination distribution $\mu \msubord \nu$ is defined by the analytic continuation of
\[
\eta_{\mu \msubord \nu} = \eta_\nu^{-1} \circ \eta_{\mu \boxtimes \nu}.
\]

\begin{Lemma}
\label{Lemma:Mult-ID}
$\mc{ID}^\circlearrowright_\ast \subset \mc{ID}^{\sutimes}_\ast$ and $\mc{ID}^\boxtimes_\ast \subset \mc{ID}^{\sutimes}_\ast$.
\end{Lemma}

\begin{proof}
Let $\set{\mu_t : t \geq 0}$ be a $\circlearrowright$-semigroup with generator $A^{\beta, \sigma}$. First, for example from Lemma~5.3 in \cite{Ans-Williams-Chernoff}, $\eta_{\mu_t}'(0) = e^{-i \beta t} e^{-t \sigma(\mf{T})} \neq 0$. Next, suppose that for some $t_0 > 0$ and $z \neq 0$, $\eta_{\mu_{t_0}}(z) = 0$. Since $A^{\beta, \sigma}$ is only zero at zero, the semigroup is the unique solution of equation \eqref{Monotone-mult-generator}, with $A^{\beta, \sigma}(0) = 0$. It follows that $\eta_{\mu_t}(z) = 0$ for $t \geq t_0$. Since $\eta$ depends analytically on $t$ \cite{Berkson-Porta}, the same holds for all $t > 0$. On the other hand, $\eta_{\mu_0}(z) = z$. We obtain a contradiction.

The second statement follows from Proposition~3.3 in \cite{Bel-Ber-Partially-mult}.
\end{proof}

\section{The class $\mc{L}$}
\label{Section:Class}

\begin{Defn}
Denote
\[
F_{\mc{L}} = \set{F: \mf{C}^+ \rightarrow \mf{C}^+ \text{ analytic} \ |\ F(z + 2 \pi) = F(z) + 2 \pi}.
\]
and
\[
\mc{L} = \set{\mu \in \mc{P}(\mf{R}) \ |\ F_\mu(z) \in F_{\mc{L}}}.
\]
\end{Defn}

\begin{Lemma}
\label{Lemma:Conditions}
An analytic function $F : \mf{C}^+ \rightarrow \mf{C}^+$ is in $F_{\mc{L}}$ if and only if
\[
F(z) = z + f(e^{iz})
\]
for  some analytic transformation $f: \mf{D} \rightarrow \mf{C}^+$.
\end{Lemma}

\begin{proof}
Clearly any function of this form is in $F_{\mc{L}}$. Conversely, if $F \in F_{\mc{L}}$, then $F(z) - z$ is $2 \pi$-periodic, and so $F(z) = z + f(e^{iz})$ for some function $f$. Moreover, it is well-known that $\Imm F\geq \Imm z$ and $F(z)-z$ is an analytic function from $\mf{C}^+$ to $\mf{C}^+$, so $f$ maps $\mf{D}$ to $\mf{C}^+$ and is analytic in the punctured disk. Since $f$ avoids the lower half plane, there are infinitely many numbers in a neighborhood of infinity which are not in the range of $f$. This implies that the singularity at zero cannot be a pole (since $f(z)$ would look like like $1/z^n$ near $0$), or an essential singularity (by Picard's theorem).  It follows that the singularity of $f$ at $0$ is removable.
\end{proof}

\begin{Lemma}\label{Cara}
There is a bijection between $F_{\mc{L}}$ (or, equivalently, $\mc{L}$) and the set of pairs
\[
\set{(\beta, \sigma) : \beta \in \mf{R}, \sigma \text{ a finite measure on } \mf{T}},
\]
given by
\[
F(z) = z - \beta + i \int_{\mf{T}} \frac{1 + \zeta e^{i z}}{1 - \zeta e^{i z}} \,d\sigma(\zeta)
\]
There is a bijection between $\set{\eta_\mu \ |\ \mu \in \mc{ID}^{\sutimes}_\ast}$ (or, equivalently, $\mc{ID}^{\sutimes}_\ast$) and the set of pairs
\[
\set{(\gamma, \sigma) : \gamma \in \mf{T}, \sigma \text{ a finite measure on } \mf{T}},
\]
given by
\[
\eta(z) = z \gamma \exp \left( - \int_{\mf{T}} \frac{1 + \zeta z}{1 - \zeta z} \,d\sigma(\zeta) \right).
\]
\end{Lemma}

\begin{proof}
By a simple modification of the Carath{\'e}odory representation (Chapter 3 of \cite{Akh65}), an analytic function $f: \mf{D} \rightarrow \mf{C}^+$ can be written as
\[
f(z) = - \beta + i \int_{\mf{T}} \frac{1 + \zeta z}{1 - \zeta z} \,d\sigma(\zeta).
\]
The first result follows. The second correspondence is simply the description of $\eta_{\nu_{\sutimes}^{\gamma, \sigma}}(z)$.
%Note also that $f(0) = - \beta + i \sigma(\mf{T})$.
\end{proof}

\begin{Prop}
\label{Prop:Atoms}
Let $\sigma$ be a measure on $\mf{T}$ with finite support. Then the measure in $\mc{L}$ corresponding to the pair $(\beta, \sigma)$ in the Lemma~\ref{Cara} is purely atomic, with countably many atoms at the solutions of the equation
\begin{equation}
\label{Atoms}
x - \beta = \int_{-\pi}^\pi \cot \frac{\theta + x}{2} \,d\sigma(e^{i \theta}),
\end{equation}
with atom at $x$ having weight
\[
\frac{1}{1 + \frac{1}{2} \int_{-\pi}^\pi (1 + \cot^2 \frac{\theta + x}{2}) \,d\sigma(e^{i \theta})}.
\]
See also Corollary~\ref{Cor:Finite-support} for a more detailed description. In the case $\sigma = \delta_{e^{i \theta}}$, the atoms are solutions of the equation
\[
x - \beta = \cot \frac{\theta + x}{2},
\]
with weights
\[
\frac{1}{\frac{3}{2} + \frac{1}{2} (x - \beta)^2}.
\]
\end{Prop}

\begin{proof}
Note first that for $\zeta = e^{i \theta}$, and $z = x$ real,
\[
i \frac{1 + e^{i (\theta + x)}}{1 - e^{i (\theta + x)}}
= - \cot \frac{\theta + x}{2}.
\]
So for real $z = x$, the function corresponding to the pair $(\beta, \sigma)$
\[
F_\mu(x) = x - \beta - \int_{-\pi}^\pi \cot \frac{\theta + x}{2} \,d\sigma(e^{i \theta})
\]
is real wherever it is finite. It follows that $\mu$ has no absolutely continuous part, but has countably many atoms at the solutions of equation \eqref{Atoms}, with the weights
\[
\frac{1}{1 + \frac{1}{2} \int_{-\pi}^\pi \csc^2 \frac{\theta + x}{2} \,d\sigma(e^{i \theta})}
= \frac{1}{1 + \frac{1}{2} \int_{-\pi}^\pi (1 + \cot^2 \frac{\theta + x}{2}) \,d\sigma(e^{i \theta})}. \qedhere
\]
%In the case $\sigma = \delta_{e^{i \theta}}$, this can be further simplified to $x - \beta = \cot \frac{\theta + x}{2}$ with the weight
%\[
%\frac{1}{1 + \frac{1}{2} (1 + \cot^2 \frac{\theta + x}{2})}
%= \frac{1}{\frac{3}{2} + \frac{1}{2} (x - \beta)^2}.
%\]
\end{proof}

\begin{Remark}
For general $\mu \in \mc{L}$,
\[
G_\mu(z) = \frac{1}{z + f(e^{iz})},
\]
and so the absolutely continuous part of the measure is
\begin{equation}
\label{Density}
d\mu(x) = \frac{1}{\pi} \lim_{r \uparrow 1} \frac{\Imm f(r e^{ix})}{(x + \Re f(r e^{ix}))^2 + (\Imm f(r e^{ix}))^2} \,dx.
\end{equation}
There are many explicit examples one can write down for specific functions $f$, perhaps the simplest one being $f(z) = z + i$. In this case,
\[
d\mu(x)
= \frac{1}{\pi} \frac{1 + \sin x}{(x + \cos x)^2 + (1 + \sin x)^2} \,dx
= \frac{1}{\pi} \frac{1 + \sin x}{x^2 + 2 x \cos x + 2 + 2 \sin x} \,dx
\]
As will be seen below, $\eta_{W(\mu)}(z) = z e^{-1} e^{iz}$.

Another important case is when $\sigma$ is the Haar measure, which corresponds to $\mu$ being a Cauchy distribution on the circle. According to Example 4.11 in \cite{Ariz-Hasebe-Semigroups}, for $c = e^{-a + i b}$, this measure on the circle is
\[
d\nu(\theta) = \frac{1}{2 \pi} \frac{1 - e^{-2a}}{1 + e^{-2a} - 2 e^{-a}  \cos(\theta - b) } \,d\theta,
\]
with $\eta_\nu(z) = c z$.
%, so its pre-image under $W$ has $F(z) = z + b + i a + 2 \pi k$ and is a Cauchy distribution.
See Example~\ref{Example:Cauchy} for more details.
\end{Remark}

\begin{Lemma}
\label{Lemma:Modulo}
\[
\mc{L} = \set{\nu : \nu \rhd \delta_{2 \pi} = \delta_{2 \pi} \rhd \nu}.
\]
In addition, for $\nu \in \mc{L}$,
\[
\delta_{2 \pi n} \uplus \nu = \delta_{2 \pi n} \rhd \nu = \nu \rhd \delta_{2 \pi n} = \delta_{2 \pi n} \boxplus \nu = \delta_{2 \pi n} \ast \nu.
\]
We can thus use the notation $\mod \delta_{2 \pi}$ for measures in $\mc{L}$ without specifying which additive convolution is being used. For measures in $\mc{L}$, all convolution arithmetic is well defined $\mod \delta_{2 \pi}$.
\end{Lemma}

\begin{proof}
For general $a$ and $\nu$,
\[
\delta_a \ast \nu = \delta_a \boxplus \nu = \nu \rhd \delta_a,
\]
since all of these are the shift of $\nu$ by $a$, and also $\delta_a \uplus \nu = \delta_a \rhd \nu$. Next we note that
\[
F_{\nu \rhd \delta_{2 \pi}}(z) = F_\nu(F_{\delta_{2 \pi}}(z)) = F_\nu(z - 2 \pi)
\]
while
\[
F_{\delta_{2 \pi} \rhd \nu}(z) = F_{\delta_{2 \pi}}(F_\nu(z)) = F_\nu - 2 \pi.
\]
These are equal precisely when $\nu \in \mc{L}$.
\end{proof}

%\begin{Lemma}
%\label{Lemma:Inverse}
%Let $F \in F_{\mc{L}}$. Then $F$ is invertible on a neighborhood of infinity, and its inverse also has the form
%\begin{equation}
%\label{Inverse-rep}
%F^{-1}(z) = z + \tilde{f}(e^{iz}),
%\end{equation}
%where $\tilde{f}$ is an analytic function from a neighborhood of $0$ to $\mf{C}^+$.
%\end{Lemma}

%\begin{proof}
%Using the properties of $F$ proved in Lemma~\ref{Lemma:Conditions}, on its domain
%\[
%F^{-1}(z + 2 \pi)  = F^{-1}(z) + 2 \pi,
%\]
%\[
%\tilde{f}(0)
%= \lim_{\Imm z \rightarrow \infty} (F^{-1}(z) - z)
%= \lim_{\Imm z \rightarrow \infty} (z - F(z)) = - f(0),
%\]
%and
%\[
%\lim_{\Imm z \rightarrow \infty} \bigl(F^{-1}(z) - z - \tilde{f}(0) \bigr) e^{-iz}
%= \lim_{\Imm z \rightarrow \infty} \bigl(z - F(z) + f(0) \bigr) e^{-i F(z)}
%= - f'(0) e^{-i f(0)}.
%\]
%It follows that the representation \eqref{Inverse-rep} holds, with $\tilde{f}$ analytic at $0$ and on the image of the domain of $F^{-1}$ under the map $z \mapsto e^{iz}$.
%\end{proof}

\begin{Lemma}
\label{Lemma:L-convolution}
$\mc{L}$ is closed under the three additive convolution operations $\uplus, \boxplus, \rhd$, under the subordination operation $\boxright$, and under Boolean additive convolution powers. Whenever $\mu \in \mc{L}$ and $\nu = \mu^{\boxplus t}$, then also $\nu \in \mc{L}$. For $\mu \in \mc{P}(\mf{R})$, and  $t>0$,
\[
\mu \in \mc{L} \Leftrightarrow \mf{B}_t(\mu) \in \mc{L},.
\]
\end{Lemma}

\begin{proof}
Let $\mu, \nu \in \mc{L}$. Then
\[
F_{\mu \rhd \nu}(z + 2 \pi) = F_\mu(F_\nu(z + 2 \pi)) = F_\mu(F_\nu(z)) + 2 \pi  = F_{\mu \rhd \nu}(z) + 2 \pi,
\]
\[
F_{\mu \uplus \nu}(z + 2 \pi) = F_\mu(z + 2 \pi) + F_\nu(z + 2 \pi) - (z + 2 \pi) = F_\mu(z) + F_\nu(z) - z + 2 \pi = F_{\mu \uplus \nu}(z) + 2 \pi,
\]
and
\[
F_{\mu^{\uplus t}}(z + 2 \pi) = t (F_\mu(z + 2 \pi) - (z + 2 \pi)) + (z + 2 \pi) = t (F_\mu(z) - z) + z + 2 \pi = F_{\mu^{\uplus t}}(z) + 2 \pi.
\]
Therefore $\mc{L}$ is closed under all these operations. Next, we observe that if $F(z + 2 \pi) = F(z) + 2 \pi$, then on its domain of definition,
\[
F^{-1}(z + 2 \pi) = F^{-1}(z) + 2 \pi.
\]
The statements about free convolution follow from this and analytic continuation. The result for the subordination distribution follows from the relation
\[
F_{\mu \boxright \nu}(z) = F_\nu^{-1}(F_{\mu \boxplus \nu}(z))
\]
and, again, analytic continuation.

Clearly, these results imply that $\mu \in \mc{L} \Rightarrow \mf{B}_t(\mu) \in \mc{L}$. Conversely, suppose $\mf{B}_t(\mu) \in \mc{L}$. Then $\mu^{\boxplus (1 + t)} \in \mc{L}$. So by the results above, $\mu \in \mc{L}$.
\end{proof}

\begin{Lemma}
\label{Lemma:Monotone-semigroups}
If $\mu \in \mc{L}$ and $\mu$ is $\rhd$-infinitely divisible, then $\mu^{\rhd t} \in \mc{L}$ for all $t > 0$. The generators of such semigroups are precisely those which are $2 \pi$-periodic.
\end{Lemma}

\begin{proof}
Let $\mu$ be $\rhd$-infinitely divisible, and denote by $\Phi$ the generator of the semigroup $\set{\mu_t = \mu^{\rhd t}}$. Denote $\nu_t = \delta_{2 \pi} \rhd \mu_t \rhd \delta_{-2\pi}$. Then $\set{\nu_t}$ is also a $\rhd$-semigroup. Since
\[
F_{\nu_t}(z) = F_{\mu_t}(z + 2 \pi) - 2 \pi,
\]
the generator of $\set{\nu_t}$ is $\tilde{\Phi}(z) = \Phi(z + 2 \pi)$. If $\mu \in \mc{L}$, then $\mu = \mu_1 = \nu_1$. But this says that $\mu_t = \nu_t$ for all $t > 0$, that is, all $\mu_t \in \mc{L}$. It also follows that $\Phi(z + 2 \pi) = \Phi(z)$. Conversely, from the definition of $F_{\mc{L}}$ it is clear that for such $\Phi$, the corresponding composition semigroup $\set{F_t : t > 0}$ is in $F_\mc{L}$.
%Conversely, from Lemma~\ref{Lemma:Conditions} it is clear that for such $\Phi$, the corresponding composition semigroup $\set{F_t : t > 0}$ is in $F_\mc{L}$.
\end{proof}

\begin{Prop}
The Bercovici-Pata bijections between $\mc{P} = \mc{ID}^{\uplus}$, $\mc{ID}^\rhd$, and $\mc{ID}^{\boxplus}$ restrict to bijections between $\mc{L}$, $\mc{L} \cap \mc{ID}^\rhd$ and $\mc{L} \cap \mc{ID}^{\boxplus}$.
%This statement also holds $\mod \delta_{2 \pi}$, which can be proven using the transformation $W$ and Bercovici-Pata bijections for multiplicative convolutions (careful, they do not hold for monotone convolution).
\end{Prop}

\begin{proof}
For $\mu \in \mc{L}$, suppose $\phi_\nu(z) = z - F_\mu(z)$. Then $\phi_\nu(z + 2 \pi) = \phi_\nu(z) + 2 \pi$, $F_\nu$ has the same property, and $\nu \in \mc{L} \cap \mc{ID}^{\boxplus}$. The converse is similar.

Next suppose again that $\mu \in \mc{L}$, and let $\Phi^+(z) = f_\mu(e^{iz})$. $\Phi^+$ maps $\mf{C}^+$ to itself, and
\[
\lim_{y \uparrow \infty} \frac{\Phi^+(iy)}{iy} = 0.
\]
So $\Phi^+$ generates a $\rhd$-semigroup $\set{\nu_t : t \geq 0}$ through
\[
\frac{d}{dt} F_{\nu_t}(z) = \Phi^+(F_{\nu_t}(z)),
\]
where $\nu = \nu_1$. Moreover, $\Phi^+$ is $2 \pi$-periodic. Therefore by Lemma~\ref{Lemma:Monotone-semigroups}, $\nu = \nu_1 \in \mc{L}$. The converse follows similarly from Lemma~\ref{Lemma:Monotone-semigroups}.
\end{proof}

%Now define $\Phi^\times(z) = z i f_\mu(z)$, so that $\Phi^\times(e^{iz}) = e^{iz} i \Phi^+(z)$. Then $\Phi^\times$ is analytic in $\mf{D}$ and maps it to the left-half-plane. Therefore it is a generator of $\circlearrowleft$-semigroup,
%\[
%\frac{d}{dt} \eta_t(z) = \Phi^\times(\eta_t(z)).
%\]
%Since all $\circlearrowleft$-infinitely-divisible measures are in $\mc{ID}^{\sutimes}$ (check), we can write $\eta_t(z) = e^{i f_t(z)}$, for $f : \mf{D} \rightarrow \mf{C}^+$ analytic. Then $F_{\nu_t}(z) = z + f_t(e^{iz}) + 2 \pi k_t$, and so all $\nu_t \in \mc{L}$. For the converse, we use Lemma~\ref{Lemma:Monotone-semigroups} check.
%
%The (check) remark is true, since we can write
%\[
%u(\eta_t(z)) = r^t e^{i \theta t} u(z)
%\]
%for some univalent function $u$. Then $u(0) = \eta_t(0) = 0$, and for any other $z$, $u(z) \neq 0$ since $u$ is univalent.

%\begin{Remark}
%A short calculation shows that more explicitly
%\[
%F_{\mu \boxplus \nu}(z) - z = f_\mu(e^{i F_{\nu \boxright \mu}(z)}) + f_\nu(e^{i F_{\mu \boxright \nu}(z)})
%\]
%so that $f_{\mu \boxplus \nu}(0) = f_\mu(0) + f_\nu(0)$ and $f'_{\mu \boxplus \nu}(0) = f_\mu'(0) + f_\nu'(0)$, and
%\[
%f_{\mu^{\boxplus t}}(e^{i z}) = t f_\mu(e^{i F_{\mu^{\boxplus (t-1)} \boxright \mu}(z)}).
%\]
%\end{Remark}

\begin{Prop}
All the elements in the class $\mc{L}$ which are not point masses are in the classical, Boolean, free, and monotone (strict) domains of attraction of the Cauchy law. Therefore their moments of all orders are undefined.
\end{Prop}

\begin{proof}
Let $\mu \in \mc{L}$. Then
\[
F_{D_{1/n} \mu^{\uplus n}}(z) - z = F_\mu(nz) - z = f_\mu(e^{i n z}) \rightarrow f_\mu(0)
\]
for every fixed $z \in \mf{C}^+$ as $n \rightarrow \infty$. So $\mu$ is in the Boolean domain of attraction of the law with the $F$-transform $f_\mu(0)$. If $f_\mu(0) \not \in \mf{R}$, this limit law is a Cauchy law. If $f_\mu(0) \in \mf{R}$, then $f_\mu(z) = f_\mu(0)$, and $\mu$ is a point mass.

The Boolean domain of attraction is known to be equal to the classical and free one (recall that the Cauchy law plays the same role in all of these theories). For the monotone case, we observe that
\[
\begin{split}
& \abs{F_{D_{1/n} \mu^{\rhd n}}(z) - z - f_\mu(0)}
= \abs{\frac{1}{n} F_\mu^{\circ n}(n z) - z - f_\mu(0)} \\
&= \frac{1}{n} \left( \abs{f_\mu(e^{i n z}) - f(0)} + \abs{f_\mu(e^{i n z} e^{i f_\mu(e^{i n z})}) - f_\mu(0)} + \abs{f_\mu(e^{i n z} e^{i f_\mu(e^{i n z} e^{i f_\mu(e^{i n z})})}) - f_\mu(0)} + \ldots \right) \\
& \leq \eps
\end{split}
\]
for any $n$ such that $\abs{f_\mu(w) - f_\mu(0)} < \eps$ for $w$ in the disk of radius $e^{- n \Imm(z)}$ around $0$.
\end{proof}

Recall that a distribution $\mu$ is called unimodal at $c$ if
\[
d\mu(x) = \mu(\set{c}) \delta_c(x) + f(x) \,dx,
\]
where $f$ is non-increasing on $[c,\infty)$ and non-decreasing on $(-\infty, c]$.

The following result characterizing unimodality in terms of the Cauchy transform  was communicated to us by Franz and Hasebe.; see also \cite{Isii-Unimodal}.

\begin{Lemma}[\cite{Franz-Hasebe-private}] \label{uni2} Let $\mu$ be a probability measure on $\real$. The following are equivalent.
\begin{enumerate}
\item $\mu$ is unimodal with mode $c$.
\item $\Imm((z-c)G_\mu'(z)) \geq0,\quad x\in\comp^+.$
\item There exists a probability measure $\nu$ on $\real$ such that $\mu=(\mathbf{u}\circledast \nu)\ast\delta_c$.
\end{enumerate}
\end{Lemma}
\begin{proof} For simplicity we assume that $c=0$.

(a) $\Leftrightarrow$ (c) is a classical characterization of unimodal distributions due to Khintchine.

(b) $\Rightarrow$ (c): Since $z G_\mu'(z)$ is a Pick function and $\lim_{y\to\infty} iy (iy G_\mu'(iy))=-1$, there exists a probability measure $\nu$ such that
$zG_\mu'(z)=-G_\nu(z)$. Integration gives us
\begin{equation}\label{mixture uniform}
G_\mu(z)= -\int_{\real\setminus\{0\}}\frac{1}{x}\log\left(\frac{z-x}{z}\right)\,\nu(dx)+ \frac{\nu(\{0\})}{z}.
\end{equation}
Since the Cauchy transform of the uniform distribution on $(0,x)$ (or $(x,0)$ if $x<0$)  is equal to $-\frac{1}{x}\log\left(\frac{z-x}{z}\right)$, we conclude that
$\mu=\mathbf{u}\circledast \nu$.

(c) $\Rightarrow$ (b): (c) implies the representation \eqref{mixture uniform}, which implies $zG_\mu'(z)=-G_\nu(z)$.
 \end{proof}

\begin{Lemma}\label{unimodality}
The unimodal distributions in class $\mathcal{L}$ are point masses and Cauchy distributions.
\end{Lemma}

\begin{proof}
Let $F=F_\mu$ and suppose, without loss of generality, that $\mu$ is unimodal at $0$. We will show that  $\Imm(F'(z))=0$ for all $z \in \mf{C}^+$. This implies that $F'(z) \in \mf{R}$ for such $z$ and, since $F'(z)$ is analytic, $F'(z)=a$ for some fixed $a\in\mathbb{R}$. We conclude that $F(z)= a z+b$, which yields the result.

Unimodality in terms of $F$ is written as $$ \Imm\Big(\frac{-zF'(z)}{F(z)^2}\Big)\geq 0.$$
Since $F\in\mathcal{L}$, for $n\in\mathbb{N}$,  $F(z+2\pi n)=F(z)+2\pi n$ and $F'(z+2\pi n)=F'(z)$. Therefore
$$0\leq \frac{\Imm(-(z+2\pi n)F'(z+2\pi n))}{(F(z+2\pi n)^2}=\frac{\Imm(-(z+2\pi n)F'(z))}{(F(z)+2\pi n)^2}.$$

Multipliying by $-2\pi n$ and taking limit as $n\to\infty$ we get $$\Imm(F'(z))=\lim_{n\to\infty}\Imm\Big(\frac{(2\pi n)(z+2\pi n)F'(z)}{(F(z)+2\pi n)^2}\Big)\leq0.$$
Similarly, taking the limit as as $n\to-\infty$
$$\Imm(F'(z))=\lim_{n\to-\infty}\Imm\Big(\frac{(2\pi n)(z+2\pi n)F'(z)}{(F(z)+2\pi n)^2}\Big)\geq0.$$

Thus $\Imm(F'(z))=0$ as desired.
\end{proof}

Since there are measures in $\mathcal{L}$ whose support is $\mathbb{R}$, Lemma \ref{unimodality} answers Problem 5.5 in \cite{Hasebe-Sakuma-Unimodality}.
\begin{Cor}
There exist measures $\mu$ whose support has a finite number of connected component and such that $\mu^{\boxplus s}$ is not unimodal for any $s>0$.
\end{Cor}

\begin{Prop}
\label{Prop:Continuity-L}
$\mc{L}$ is closed under weak limits.
\end{Prop}
\begin{proof}
Let $\set{\mu_n}_{n=1}^\infty \subset \mc{L}$ and $\mu_n \rightarrow \mu$ weakly. Thus $F_{\mu_n} \rightarrow F_\mu$ uniformly on compact subsets of $\mf{C}^+$, and so $f_{\mu_n} \rightarrow f_\mu$ uniformly on compact subsets of $\mf{D} \setminus \set{0}$. In particular $f_\mu$ is analytic in the punctured disk. As in the proof of Lemma~\ref{Lemma:Conditions}, we conclude that the singularity of $f_\mu$ at $0$ is removable, and we may assume it is analytic in $\mf{D}$. Therefore $\mu \in \mc{L}$.
\end{proof}

\begin{Remark}
$\mc{ID}^{\sutimes}_\ast$ is \emph{not} closed under weak limits. Indeed, if $\eta_{\nu_t} = e^{-t} z$, so that $\nu_t$ is a wrapped Cauchy distribution from Example~\ref{Example:Cauchy}, then as $t \rightarrow \infty$, $\eta_{\nu_t} \rightarrow 0$, which is the $\eta$-transform of the Haar measure on $\mf{T}$.

On the other hand, it is not difficult to see that the full $\mc{ID}^{\sutimes}$ is closed under weak limits.
\end{Remark}

\section{The transformation}
\label{Section:Transformation}

\subsection{Homomorphism properties of $W$}

\begin{Prop}
\label{Prop:A-on-functions}
The map
\[
L : F_{\mc{L}} \rightarrow \set{\eta_\mu \ |\ \mu \in \mc{ID}^{\sutimes}_\ast}
\]
determined by
\[
\exp(i F(z)) = L(F)(e^{iz}) = \eta(e^{iz})
\]
is well-defined and onto. The pre-image of each $\eta$ is the equivalence class of functions $F$ modulo the equivalence relation $F \sim F + 2 \pi$.
\end{Prop}

\begin{proof}
We verify that for $F(z) = z + f(e^{iz})$,
\[
L(F)(z) = \eta(z) = z e^{i f(z)},
\]
and satisfies all the properties on the right-hand-side of \eqref{Multi-Boolean}. Conversely, any $\eta$ from the set on the right-hand-side is of this form for some analytic $f : \mf{D} \rightarrow \mf{C}^+$, determined up to an additive integer multiple of $2 \pi$.
\end{proof}

\begin{Example}
\label{Example:Cauchy}
The Cauchy distribution on $\mf{R}$ is
\[
d\mu_t(x) = \frac{1}{\pi} \frac{t}{x^2 + t^2} \,dx.
\]
Its $F$-transform is $F_{\mu_t}(z) = z + t i$, and its image $\eta_t(z) = L(F_{\mu_t}(z)) = e^{-t} z$. Then $\eta_t = \eta_{\nu_t}$, for
\[
d\nu_t(e^{ix}) = \frac{1}{2\pi} \frac{1 - e^{-2t}}{\abs{\zeta - e^{-t}}^2} \,dx.
\]
As first observed in Section 5.2 of \cite{BiaProcesses}, $\nu_t$ is a wrapping of $\mu_t$.
\end{Example}

\begin{Thm}
\label{Thm:Wrapping}
The wrapping map $W$ from Definition~\ref{Defn:Wrapping} maps $\mc{L}$ onto $\mc{ID}^{\sutimes}_\ast$. Considered as a map
\[
W : \mc{L} \rightarrow \mc{ID}^{\sutimes}_\ast,
\]
it satisfies
\[
L(F_\mu) = \eta_{W(\mu)},
\]
in other words
\[
\exp(i F_{\mu}(z)) = \eta_{W(\mu)}(e^{iz}).
\]
The pre-image of each $\nu$ is the equivalence class of measures $\mu$ modulo the equivalence relation $\mod \delta_{2 \pi}$.
\end{Thm}

\begin{proof}
Since
\[
d\mu(x)
= - \frac{1}{\pi} \lim_{y \downarrow 0} \Imm \frac{1}{F_\mu(x + i y)} \,dx
= \frac{1}{\pi} \lim_{y \downarrow 0} \frac{\Imm F_\mu(x + i y)}{\abs{F_\mu(x + i y)}^2} \,dx,
\]
for $\mu \in \mc{L}$
\[
d(W(\mu))(e^{ix})
= \frac{1}{\pi} \lim_{y \downarrow 0} \sum_{n \in \mf{Z}} \frac{\Imm F_\mu(x + i y + 2 \pi n)}{\abs{F_\mu(x + i y + 2 \pi n)}^2} \,dx
= \frac{1}{\pi} \lim_{y \downarrow 0} \sum_{n \in \mf{Z}} \frac{\Imm F_\mu(x + i y)}{\abs{F_\mu(x + i y) + 2 \pi n)}^2} \,dx.
\]
We now apply the Poisson summation formula for the Cauchy distribution, to get
\[
\begin{split}
d(W(\mu))(e^{ix})
& = \frac{1}{2\pi} \lim_{y \downarrow 0} \sum_{n \in \mf{Z}} e^{i n \Re F_\mu(x + i y)} e^{- \abs{n} \Imm F_\mu(x + i y)} \,dx \\
& = \frac{1}{2\pi} \lim_{y \downarrow 0} \left(1 + \sum_{n=1}^\infty \left( e^{i n F_\mu(x + i y)} + e^{- i n \overline{F_\mu (x + i y)}} \right) \right) \,dx\\
& = \frac{1}{2\pi} \lim_{y \downarrow 0} \left( \frac{1}{1 - e^{i F_\mu(x + i y)}} + \frac{1}{1 - e^{- i \overline{F_\mu (x + i y)}}} - 1 \right) \,dx \\
& = \frac{1}{2\pi} \lim_{y \downarrow 0} \frac{1 - e^{-2 \Imm F_\mu(x + i y)}}{\abs{1 - e^{i F_\mu(x + i y)}}^2} \,dx.
\end{split}
\]
On the other hand, letting $\nu$ be determined by $L(F_\mu) = \eta_{\nu}$ and $r = e^{-y}$,
\[
\begin{split}
d\nu(e^{-ix})
& = \frac{1}{2\pi} \lim_{r \uparrow 1} \Re \frac{1 + \eta_{\nu}(r e^{ix})}{1 - \eta_{\nu}(r e^{ix})} \,dx
= \frac{1}{2\pi} \lim_{y \downarrow 0} \Re \frac{1 + \eta_{\nu}(e^{-y} e^{ix})}{1 - \eta_{\nu}(e^{-y} e^{ix})} \,dx \\
& = \frac{1}{2\pi} \lim_{y \downarrow 0} \Re \frac{1 + \eta_{\nu}(e^{i (x + i y)})}{1 - \eta_{\nu}(e^{i (x + i y)})} \,dx
= \frac{1}{2\pi} \lim_{y \downarrow 0} \Re \frac{1 + e^{i F_\mu(x + i y)}}{1 - e^{i F_\mu(x + i y)}} \,dx \\
& = \frac{1}{2\pi} \lim_{y \downarrow 0} \frac{1 - e^{-2 \Imm F(x + i y)}}{\abs{1 - e^{i F_\mu(x + i y)}}^2} \,dx.
\end{split}
\]
Therefore $\nu = W(\mu)$. Finally, for the last statement, we note that $F_{\delta_{2 \pi} \uplus \mu }(z) = F_\mu(z) + 2 \pi$.
\end{proof}

\begin{Example}
\label{Example:Counter}
Let $\mu = \frac{1}{2} (\delta_{-2 \pi} + \delta_{2 \pi})$ be a Bernoulli distribution. Then $W(\mu) = \delta_1$. Also, $\delta_1 \boxtimes \delta_1 = \delta_1$, while $\mu \boxplus \mu$ is an arcsine distribution. Thus $W(\mu) \boxtimes W(\mu) \neq W(\mu \boxplus \mu)$.
\end{Example}

\begin{Thm}
\label{Thm:Homomorphism}
For any $\mu_1, \mu_2 \in \mc{L}$,
\[
W(\mu_1) \circlearrowleft W(\mu_2) = W(\mu_1 \rhd \mu_2),
\]
\[
W(\mu_1) \utimes W(\mu_2) = W(\mu_1 \uplus \mu_2),
\]
\[
W(\mu_1) \boxtimes W(\mu_2) = W(\mu_1 \boxplus \mu_2),
\]
and
\[
W(\mu_2) \msubord W(\mu_1) = W(\mu_2 \boxright \mu_1).
\]
Conversely, for any $\nu_1, \nu_2 \in \mc{ID}_*^{\sutimes}$,
\[
W^{-1}(\nu_1 \circlearrowleft \nu_2) = W^{-1}(\nu_1) \rhd W^{-1}(\nu_2) \mod \delta_{2 \pi},
\]
\[
W^{-1}(\nu_1 \utimes \nu_2) = W^{-1}(\nu_1) \uplus W^{-1}(\nu_2) \mod \delta_{2 \pi},
\]
and
\[
W^{-1}(\nu_1 \boxtimes \nu_2) = W^{-1}(\nu_1) \boxplus W^{-1}(\nu_2) \mod \delta_{2 \pi}.
\]
Thus $(\mc{L} \mod \delta_{2\pi})$ and $\mc{ID}^{\sutimes}_*$ are isomorphic as semigroups with the respect to the additive, respectively, multiplicative, Boolean, free, and monotone convolutions.
\end{Thm}

\begin{proof}
\[
\begin{split}
\eta_{W(\mu_1) \circlearrowleft W(\mu_2)}(e^{iz}))
& = \eta_{W(\mu_1)}(\eta_{W(\mu_2)}(e^{iz}))
= \eta_{W(\mu_1)}(e^{i F_{\mu_2}(z)}) \\
& = \exp(i F_{\mu_1}(F_{\mu_2}(z)))
= \exp(i F_{\mu_1 \rhd \mu_2}(z))
= \eta_{W(\mu_1 \rhd \mu_2)}(e^{iz}),
\end{split}
\]
which implies the first identity. The second identity follows from the observation that
\[
\exp(i (F_{\mu}(z) - z)) = \frac{\eta_{W(\mu)}(e^{iz})}{e^{iz}},
\]
by a similar argument. On a neighborhood of infinity,
\[
\exp(i F_{\mu}^{-1}(z + 2 \pi n)) = \eta_{W(\mu)}^{-1}(e^{iz}),
\]
\[
\exp(i F_{\mu}^{-1}(z)) = \eta_{W(\mu)}^{-1}(e^{i(z - 2 \pi n)}) = \eta_{W(\mu)}^{-1}(e^{i z}),
\]
\[
\exp(i \varphi_{\mu}(z)) = \Sigma_{W(\mu)}(e^{iz}),
\]
which implies the third identity. Finally,
\[
\begin{split}
\eta_{W(\mu_1)}(\eta_{W(\mu_2) \msubord W(\mu_1)}(e^{i z}))
& = \eta_{W(\mu_1) \boxtimes W(\mu_2)}(e^{i z})
= \eta_{W(\mu_1 \boxplus \mu_2)}(e^{i z}) \\
& = \exp(i F_{\mu_1 \boxplus \mu_2}(z))
= \exp(i F_{\mu_1}(F_{\mu_2 \boxright \mu_1}(z)) \\
& = \eta_{W(\mu_1)}(\eta_{W(\mu_2 \boxright \mu_1)}(e^{i z})).
\end{split}
\]
Since $\eta_{W(\mu_1)}$ is injective on a sufficiently small disk, the result for subordination distributions follows.

The second set of statements follows from Lemma~\ref{Lemma:Modulo}.
%For future reference, we record the following relation between generators: if
%\[
%\Phi^+(z) = \left. \frac{d}{dt} \right|_{t = 0} F_{\mu_t}(z) = \left. \frac{d}{dt} \right|_{t = 0} f_{\mu_t}(e^{iz}),
%\]
%and
%\[
%\Phi^\times(z) =  \left. \frac{d}{dt} \right|_{t = 0} \eta_{W(\mu_t)}(z),
%\]
%then
%\[
%\Phi^\times(e^{iz}) = \left. \frac{d}{dt} \right|_{t = 0} e^{i F_{\mu_t}(z)} = e^{iz} i \Phi^+(z),
%\]
%or equivalently
%\[
%\frac{\Phi^\times(z)}{z} = i \left. \frac{d}{dt} \right|_{t = 0} f_{\mu_t}(z)
%\]
\end{proof}

\begin{Prop}
\label{Prop:ID-preimage}
Recall that multiplicative convolution powers are in general multi-valued. This ambiguity can be explained using transformation $W$. Let $\mu \in \mc{L}$. Then for any $t \geq 0$,
\[
W(\mu^{\uplus t}) = W(\mu)^{\sutimes t},
\]
and all values of $W(\mu)^{\sutimes t}$ arise in this way. Similarly, whenever $\mu^{\boxplus t}$ is defined,
\[
W(\mu^{\boxplus t}) = W(\mu)^{\boxtimes t},
\]
and all values of $W(\mu)^{\boxtimes t}$ arise in this way.

$W$ maps $\boxplus$, $\uplus$, and $\rhd$-infinitely divisible distributions in $\mc{L}$, and the corresponding semigroups, to their multiplicative counterparts. If $\nu \in \mc{ID}^\boxtimes_\ast$, then every element of $W^{-1}(\nu)$ is in $\mc{L} \cap \mc{ID}^\boxplus$. If $\nu \in \mc{ID}^\circlearrowright_\ast$, then there is $\mu \in \mc{L} \cap \mc{ID}^\rhd_\ast$ such that $W(\mu) = \nu$.
\end{Prop}

\begin{proof}
The proof of the first two statements is similar to Theorem~\ref{Thm:Homomorphism}. For the ambiguity, note that if $\nu = W(\mu)$, then also $\nu = W(\mu \uplus \delta_{2 \pi n})$, and so $\nu^{\sutimes t} = W(\mu^{\uplus t}) \utimes \delta_{e^{2 \pi i n t}}$ for different $n \in \mf{Z}$. These are precisely all the possible values of $\nu^{\sutimes t}$. The argument for $\nu^{\boxtimes t}$ is similar. The statement about convolution semigroups also follows as in Theorem~\ref{Thm:Homomorphism}.

Let $\nu = W(\mu)$ with $\nu \in \mc{ID}^\boxtimes_\ast$. Then $\Sigma_\nu$ has an analytic continuation to a map $\mf{D} \rightarrow \mf{C} \setminus \mf{D}$. So $e^{i \phi_\mu(z)} = \Sigma_\nu(e^{i z})$ has an analytic continuation to a map $\mf{C}^+ \rightarrow \mf{C} \setminus \mf{D}$. Thus it is an analytic map from a simply connected domain to a domain not containing zero. It follows that
\[
\phi_\mu = - i \log \Sigma_\nu(e^{i z}) + 2 \pi k
\]
can be defined through analytic continuation, with values in $\mf{C}^-$. Thus $\mu = \widetilde{\mu} \boxplus \delta_{2 \pi k}$ for some $\widetilde{\mu} \in \mc{ID}^\boxplus$. But this implies $\mu \in \mc{ID}^\boxplus$.

Now suppose $\nu \in \mc{ID}^\circlearrowright_\ast$, so for some analytic generator $\Phi^\times : \mf{D} \rightarrow i \mf{C}^+$, \[
\eta_{\nu_t}(z) \Phi^\times(\eta_{\nu_t}(z)) = \frac{d}{dt} \eta_{\nu_t}(z),
\]
with $\nu_1 = \nu$. Define analytic $\Phi^+ : \mf{C}^+ \rightarrow \mf{C}^+$ by
\begin{equation}
\label{Phi-plus-times}
\Phi^+(z) = - i \Phi^\times(e^{iz}).
\end{equation}
Then $\Phi^+$ is $2 \pi$-periodic, and
\[
\lim_{y \uparrow \infty} \frac{\Phi^+(x + i y)}{i y}  = \lim_{y \uparrow \infty} \frac{- i \Phi^\times(0)}{y} = 0.
\]
So $\Phi^+$ generates a $\rhd$-semigroup $\set{\mu_t} \subset \mc{L}$ by
\[
\Phi^+(F_{\mu_t}(z)) = \frac{d}{dt} F_{\mu_t}(z).
\]
Moreover,
\[
\frac{d}{dt} e^{i F_{\mu_t}(z)}
= i e^{i F_{\mu_t}(z)} \frac{d}{dt} F_{\mu_t}(z)
= i e^{i F_{\mu_t}(z)} \Phi^+(F_{\mu_t}(z))
= e^{i F_{\mu_t}(z)} \Phi^\times(e^{i F_{\mu_t}(z)}),
\]
and so $e^{i F_{\mu_t}(z)} = \eta_{\nu_t}(e^{i z})$ and $\nu_t = W(\mu_t)$.
\end{proof}

\begin{Cor}
$\mc{ID}^{\sutimes}_\ast$ is closed under the three multiplicative convolution operations $\utimes, \boxtimes, \circlearrowright$, under the subordination operation $\msubord$, and under multiplicative Boolean convolution powers. Whenever $\mu \in \mc{ID}^{\sutimes}_\ast$ and $\nu = \mu^{\boxtimes t}$, then also $\nu \in \mc{ID}^{\sutimes}_\ast$.
\end{Cor}

\begin{proof}
Closure under Boolean convolution, Boolean powers, and monotone convolution follows immediately from definition \eqref{Convolutions} and characterization \eqref{Multi-Boolean}.

The claims for the free convolution, and for subordination, follow by combining Lemma~\ref{Lemma:L-convolution} with Theorem~\ref{Thm:Homomorphism} and Proposition~\ref{Prop:ID-preimage}. For example, if $\nu_1, \nu_2 \in \mc{ID}^{\sutimes}_\ast$, and $\mu_1, \mu_2 \in \mc{L}$ are chosen so that $\nu_1 = W(\mu_1)$, $\nu_2 = W(\mu_2)$, then $\mu_1 \boxplus \mu_2 \in \mc{L}$ and $\nu_1 \boxtimes \nu_2 = W(\mu_1 \boxplus \mu_2) \in \mc{ID}^{\sutimes}_\ast$.
\end{proof}

\begin{Remark}
\label{Remark:Cebron}
In \cite{Cebron-Multiplicative}, C{\'e}bron defined a homomorphism
\[
\mb{e}_\boxplus : \mc{ID}^\boxplus \rightarrow \mc{ID}^{\boxtimes}
\]
which satisfies
\begin{equation}
\label{BP}
W = \mb{e}_\ast = \mc{BP}_{\boxtimes \rightarrow \circledast} \circ \mb{e}_\boxplus \circ \mc{BP}_{\ast \rightarrow \boxplus}.
\end{equation}
Here $\mc{BP}$ denote various Bercovici-Pata maps, and $\mb{e}_\ast : \mc{ID}^\ast \rightarrow \mc{ID}^\circledast$ is simply $\mb{e}_\ast = W$ (in fact, C{\'e}bron used the counterclockwise wrapping, but his arguments are not affected by this slight change of definition). He also proved that for $\mu \in \mc{ID}^\boxplus$,
\[
\mb{e}_\boxplus (\mu) = \lim_{n \rightarrow \infty} \left( \mb{e}_\ast(\mu^{\boxplus \frac{1}{n}}) \right)^{\boxtimes n}.
\]
By combining this with Theorem~\ref{Thm:Homomorphism}, it follows that for $\mu \in \mc{L} \cap \mc{ID}^\boxplus$, $\mb{e}_\boxplus(\mu) = W(\mu)$, and in particular in this case the limit above is unnecessary.

A calculation in Section~2.3 and Proposition~3.1 of \cite{Cebron-Multiplicative} also shows that for $(\alpha, \tau)$ appearing in the L{\'e}vy-Khinchin representation \eqref{LH-add} of $\mu$, and $(\gamma, \sigma)$ appearing in the (multiplicative) L{\'e}vy-Khinchin representation \eqref{LH-Mult} of $W(\mu)$, we have the relations
\begin{equation}
\label{Wrapping-Levy}
\frac{1}{1 - \Re(\zeta)} \,d\sigma(\zeta)|_{\mf{T} \setminus \set{1}} = W \left( \frac{x^2 + 1}{x^2} \,d\tau(x) |_{\mf{R} \setminus \set{0}} \right), \quad \sigma(\set{1}) = \frac{1}{2} \tau(\set{0})
\end{equation}
and
\begin{equation}
\label{Wrapping-Levy2}
\gamma = e^{-i \alpha} \exp \left(- i \int_{\mf{R} \setminus \set{0}} \left( \sin x - \frac{x}{1 + x^2} \right) \frac{1 + x^2}{x^2} \,d\tau(x) \right).
\end{equation}
These formulas look more natural in the language of canonical triples rather than pairs.
\end{Remark}

\begin{Prop}
\label{Prop:BP-intertwiner}
Define the map
\[
\set{(\alpha, \tau) : \alpha \in \mf{R}, \tau \text{ a finite measure on } \mf{R}, \nu^{\alpha, \tau}_{\uplus} \in \mc{L}}
\rightarrow \set{(\gamma, \sigma) : \gamma \in \mf{T}, \sigma \text{ a finite measure on } \mf{T}}
\]
via the relations \eqref{Wrapping-Levy} and \eqref{Wrapping-Levy2}. Then
\begin{equation}
\label{Sigma-Tau}
e^{- i \alpha} \exp \left( i \int_{\mf{R}} \frac{1 + x z}{x - z} \,d\tau(t) \right)
= \gamma \exp \left( - \int_{\mf{T}} \frac{1 + \zeta e^{iz}}{1 - \zeta e^{iz}} \,d\sigma(\zeta) \right).
\end{equation}
For such $(\alpha, \tau)$ we have the relations
\begin{align}
\label{W-ast}
W(\nu_{\ast}^{\alpha, \tau}) & = \nu_{\circledast}^{\gamma, \sigma}, \\
\label{W-box}
W(\nu_{\boxplus}^{\alpha, \tau}) & = \nu_{\boxtimes}^{\gamma, \sigma}, \\
\label{W-u}
W(\nu_{\uplus}^{\alpha, \tau}) & = \nu_{\sutimes}^{\gamma, \sigma},
\end{align}
and
\begin{equation}
\label{W-rhd}
W(\nu_{\rhd}^{\alpha, \tau}) = \nu_{\circlearrowright}^{\beta, \sigma},
\end{equation}
where
\begin{equation}
\label{Beta}
- i \alpha + i \int_{\mf{R}} \frac{1 + x z}{x-z} \,d\tau(x)
= - i \beta - \int_{\mf{T}} \frac{1 + \zeta e^{i z}}{1 - \zeta e^{i z}} \,d\sigma(\zeta).
\end{equation}
%\[
%\beta =  \alpha + \int_{\mf{R} \setminus \set{0}} \left( \sin x - \frac{x}{1 + x^2} \right) \frac{1 + x^2}{x^2} \,d\tau(x).
%\]
As a result, $W$ intertwines the Bercovici-Pata maps (the multiplicative ones are not bijections)
\begin{align}
W \circ \mc{BP}_{\boxplus \rightarrow \ast} & = \mc{BP}_{\boxtimes \rightarrow \circledast} \circ W \qquad \text{ on } \mc{L} \cap \mc{ID}^\boxplus, \label{BP:free} \\
W \circ \mc{BP}_{\uplus \rightarrow \ast} & = \mc{BP}_{\sutimes \rightarrow \circledast} \circ W \qquad \text{ on } \mc{L}, \label{BP-Boolean}\\
W \circ \mc{BP}_{\rhd \rightarrow \ast} & = \mc{BP}_{\circlearrowright \rightarrow \circledast} \circ W \qquad \text{ on } \mc{L} \cap \mc{ID}^\rhd \label{BP-Monotone}
\end{align}
as well as $\mc{BP}_{\boxplus \rightarrow \uplus}$, $\mc{BP}_{\uplus \rightarrow \boxplus}$, $\mc{BP}_{\rhd \rightarrow \boxplus}$, $\mc{BP}_{\rhd \rightarrow \uplus}$, on the appropriate subsets of $\mc{L}$, with their multiplicative counterparts.
\end{Prop}

\begin{proof}
Relation \eqref{BP:free} follows from relation~\eqref{BP} and the discussion following it, and equation \eqref{W-ast} follows from the last comment in Remark~\ref{Remark:Cebron}. Combining them, we obtain equation~\eqref{W-box}. Next, the left-hand side of equation \eqref{Sigma-Tau} is $\exp(i (\phi_{\nu_\boxplus^{\alpha, \tau}}(z)))$, while the right-hand side is $\Sigma_{\nu_\boxtimes^{\gamma, \sigma}}(e^{i z})$, and these were just shown to be equal. Reinterpreting them as $\exp(i (F_{\nu_\uplus^{\alpha, \tau}}(z) - z))$ and $\eta_{\nu_{\sutimes}^{\gamma, \sigma}}(e^{i z})$, we obtain \eqref{W-u} and \eqref{BP-Boolean}. Finally, combining \eqref{LH-Monotone} and \eqref{Monotone-mult-generator} with equation~\eqref{Phi-plus-times} from the proof of Proposition~\ref{Prop:ID-preimage}, we see that \eqref{W-rhd} holds with $\beta$ as in equation \eqref{Beta}. Taking $\gamma = e^{-i \beta}$, we again obtain relation \eqref{Sigma-Tau}, and so also \eqref{BP-Monotone}. The final statement follows from relations \eqref{W-box}, \eqref{W-u}, \eqref{W-rhd}.
\end{proof}

\begin{Remark}
Let $\nu$ be the multiplicative free Gaussian measure, with
\[
\Sigma_\nu(z) = \exp \left( \frac{1}{2} \frac{1 + z}{1 - z} \right).
\]
It is $\boxtimes$-infinitely divisible, with free canonical pair $(0, \frac{1}{2} \delta_1)$. So its pre-image under C{\'e}bron's map consists of measures with free canonical pairs $(\alpha, \tau)$ such that
\[
\begin{split}
& W \left( \frac{x^2 + 1}{x^2} \,d\tau(x) |_{\mf{R} \setminus \set{0}} \right) = 0, \quad \tau(\set{0}) = 1, \\
& e^{i \alpha} = \exp \left(- i \int_{\mf{R} \setminus \set{0}} \left( \sin x - \frac{x}{1 + x^2} \right) \frac{1 + x^2}{x^2} \,d\tau(x) \right).
\end{split}
\]
That is, $\tau$ is supported on $\set{2 \pi k \ |\ k \in \mf{Z}}$, $\tau(\set{0}) = 1$, and $\alpha = 2 \pi n + \sum_{k \in \mf{Z} \setminus \set{0}} \frac{1}{2 \pi k} \tau(\set{2 \pi k})$. These, of course, include the semicircular distribution, with $\alpha = 0$, $\tau = \delta_0$. On the other hand, only some of these are in $\mc{L}$. Indeed, let $\mu \in \mc{L}$ be the measure with
\[
\phi_\mu(z) = - i \frac{1}{2} \frac{1 + e^{iz}}{1 - e^{iz}} + 2 \pi n.
\]
Then $e^{i F_\mu(z)} = \eta_\nu(e^{iz})$. So by Theorem~\ref{Thm:Wrapping}, $W(\mu) = \nu$. Also, $\phi_\mu : \mf{C}^+ \rightarrow \mf{C}^-$, and so $\mu \in \mc{ID}^\boxplus$ (and also is well-defined). Its free canonical pair $(\alpha, \tau)$ is determined by
\begin{equation}
\label{Free-Gaussian}
\phi_{\mu}(z) =  \alpha + \int_{\mf{R}} \frac{1 + x z}{z-x} \,d\tau(x) = - i \frac{1}{2} \frac{1 + e^{iz}}{1 - e^{iz}} + 2 \pi n.
\end{equation}
Since
\[
\lim_{y \downarrow 0} i y \phi_\mu(x + i y) = (1 + x^2) \tau(\set{x}),
\]
from equation~\eqref{Free-Gaussian} it follows that
\[
d\tau(x) = \sum_{k \in \mf{Z}} \frac{1}{1 + (2 \pi k)^2} \delta_{2 \pi k}(x)
\]
and we may take
\[
\alpha = 2 \pi n + \sum_{k \neq 0} \frac{1}{2 \pi k (1 + (2 \pi k)^2)}.
\]
These are exactly the pre-images of $\nu$ in $\mc{L}$.
\end{Remark}

\begin{Remark}
$W$ also transforms additive conditionally free and monotone convolutions into their multiplicative counterparts. The constructions are similar to those already discussed, and these operations do not arise in our applications, so we only outline the argument. Conditionally free and monotone convolutions operate on pairs of measures. For $\mu, \nu \in \mc{P}(\mf{R})$, define as in \cite{Wang-Additive-c-free}
\[
\phi_{(\mu, \nu)}(z) = F_\nu^{-1}(z) - F_\mu(F_\nu^{-1}(z))
\]
and the additive conditionally free convolution by $(\mu_1, \nu_1) \boxplus_c (\mu_2, \nu_2) = (\widetilde{\mu}, \nu_1 \boxplus \nu_2)$ and
\[
\phi_{(\mu_1, \nu_1) \boxplus_c (\mu_2, \nu_2)}(z) = \phi_{\mu_1, \nu_1}(z) + \phi_{\mu_2, \nu_2}(z)
\]
Similarly, following \cite{Popa-Wang-Multiplicative-c-free} up to an inversion, define for $\mu, \nu \in \mc{P}(\mf{T})$,
\[
\Sigma_{(\mu, \nu)}(z) = \frac{\eta_\mu(\eta_\nu^{-1}(z))}{\eta_\nu^{-1}(z)}
\]
and the multiplicative conditionally free convolution by $(\mu_1, \nu_1) \boxtimes_c (\mu_2, \nu_2) = (\widetilde{\mu}, \nu_1 \boxtimes \nu_2)$ and
\[
\Sigma_{(\mu_1, \nu_1) \boxtimes_c (\mu_2, \nu_2)}(z) = \Sigma_{\mu_1, \nu_1}(z) \Sigma_{\mu_2, \nu_2}(z)
\]
Then clearly if $(\mu_1, \nu_1) \boxplus_c (\mu_2, \nu_2) = (\widetilde{\mu}, \widetilde{\nu})$, then
\[
(W(\mu_1), W(\mu_2)) \boxtimes_c (W(\mu_2), W(\nu_2)) = (W(\widetilde{\mu}), W(\widetilde{\nu})).
\]
On the other hand, following \cite{Hasebe-Conditionally-monotone}, define the additive conditionally monotone convolution by $(\mu_1, \nu_1) \rhd_c (\mu_2, \nu_2) = (\widetilde{\mu}, \nu_1 \rhd \nu_2)$ and
\[
F_{\widetilde{\mu}}(z) = F_{\mu_1}(F_{\nu_2}(z)) + F_{\mu_2}(z) - F_{\nu_2}(z);
\]
while following \cite{Hasebe-Conditionally-monotone-2}, define the multiplicative conditionally monotone convolution by
\[
(\mu_1, \nu_1) \circlearrowright_c (\mu_2, \nu_2) = (\widetilde{\mu}, \nu_1 \circlearrowright \nu_2)
\]
and
\[
\eta_{\widetilde{\mu}}(z) = \frac{\eta_{\mu_2}(z)}{\eta_{\nu_2}(z)} \eta_{\mu_1}(\eta_{\nu_2}(z)).
\]
Again, it is clear that for $(\mu_1, \nu_1) \rhd_c (\mu_2, \nu_2) = (\widetilde{\mu}, \widetilde{\nu})$, we have
\[
(W(\mu_1), W(\mu_2)) \circlearrowright_c (W(\mu_2), W(\nu_2)) = (W(\widetilde{\mu}), W(\widetilde{\nu})).
\]
Presumably the same property also holds for Hasebe's convolution for triples of measures from \cite{Hasebe-triple-product}, but we have not verified the details.

\end{Remark}

\subsection{Other properties of $W$}

\begin{Prop}
\label{Prop:Atoms0}
If $\mu \in \mc{L}$ has an atom at $x$, then $W(\mu)$ has an atom at $e^{-ix}$ of the same size. Conversely, if $e^{-ix}$ is an atom of $W(\mu)$, there is a unique $n \in \mf{Z}$ such that $\mu$ has an atom at $x + 2 \pi n$, and it is of the same size.
\end{Prop}

\begin{proof}
For $x \in \mf{R}$,
\[
\mu(\set{x}) = \lim_{y \downarrow 0} \frac{i y}{F_\mu(x + i y)},
\]
while for $e^{-ix} \in \mf{T}$,
\[
\nu(\set{e^{-ix}})
= \lim_{r \uparrow 1} (1-r) \frac{\eta_\nu(r e^{ix})}{1 - \eta(r e^{ix})}
= \lim_{y \downarrow 0} (1-e^{-y}) \frac{\eta_\nu(e^{-y} e^{ix})}{1 - \eta(e^{-y} e^{ix})},
\]
where we set $r = e^{-y}$. If $\nu = W(\mu)$, so that $\eta_\nu(e^{iz}) = e^{i F_\mu(z)}$, then
\[
\nu(\set{e^{-ix}})
= \lim_{y \downarrow 0} (1-e^{-y}) \frac{\exp(i F_\mu(x + i y))}{1 - \exp(i F_\mu(x + i y))}
= \lim_{y \downarrow 0}  \frac{(1-e^{-y})}{\exp(-i F_\mu(x + i y)) - 1}
\]
In particular, if $x$ is an atom of $\mu$, then $\lim_{y \downarrow 0} F_\mu(x + i y) = 0$. Therefore in this case,
\begin{equation}
\label{Ratio}
\frac{\mu(\set{x})}{\nu(\set{e^{-ix}})}
= \lim_{y \downarrow 0} \frac{i y}{(1-e^{-y})} \frac{\exp(-i F_\mu(x + i y)) - 1}{F_\mu(x + i y)} = 1,
\end{equation}
and so $\nu$ has an atom at $e^{-ix}$ of weight $\mu(\set{x})$.

For the converse, if $e^{-ix}$ is an atom of $\nu$, then $\lim_{y \downarrow 0} \exp(-i F_\mu(x + i y)) = 1$. So there is a unique (by continuity)  $n \in \mf{Z}$ such that $\lim_{y \downarrow 0} (- F_\mu(x + i y)) = 2 \pi n$, i.e. (since $\mu \in \mc{L}$)
\[
\lim_{y \downarrow 0} F_\mu(x + 2 \pi n + i y)) = 0.
\]
We now use equation \eqref{Ratio} again.
\end{proof}

\begin{Cor}
\label{Cor:L-atoms-1}
If $\mu \in \mc{L}$, at most one element of each $\set{x + 2 \pi n \ |\ n \in \mf{Z}}$ may be an atom of $\mu$.
\end{Cor}

%For $\mu$ purely atomic, from the preceding proposition we get the description of atoms of $W(\mu)$ which add up to $1$, and therefore completely describe $W(\mu)$ (up to reverse direction?) For general $\mu$, the result follows by taking weak limits. Is every measure in $\mc{L}$ a weak limit of purely atomic measures \emph{in $\mc{L}$}?

See Corollary~\ref{Cor:L-atoms-2} for a follow-up.

In particular, for the case of discrete measure $\sigma$ as in Proposition~\ref{Prop:Atoms}, we have a complete description of $\nu_{\sutimes}^{\beta, \sigma}$. The following corollary generalizes Example 3.8 of \cite{Franz-Boolean-circle} and the results in Section 2.4 of \cite{Hamdi-Unitary-BM}.

\begin{Cor}
\label{Cor:Finite-support}
Let $\sigma$ be a measure on $\mf{T}$ with finite support, as in Proposition~\ref{Prop:Atoms},
\[
\sigma = \sum_{j=1}^N a_j \delta_{e^{i \theta_j}},
\]
where $0 \leq \theta_1 < \ldots < \theta_N < 2 \pi$. Then the measure $\nu_{\sutimes}^{e^{i \beta}, \sigma} \in \mc{ID}^{\sutimes}_\ast$ is purely atomic, and its atoms can be decomposed into $N$ families $\set{(e^{x_{jk}})_{j \in \mf{Z}}, 1 \leq k \leq N}$, with $x_{jk}$ the unique solution of the equation
\begin{equation}
%\label{Atoms}
x + \beta = \sum_{j=1}^N a_j \cot \frac{x - \theta_j}{2}
\end{equation}
in the interval
\[
\begin{cases}
(2k \pi + \theta_j, 2k \pi + \theta_{j+1}), & 1 \leq j < N, \\
(2k \pi + \theta_N, 2(k+1) \pi + \theta_{1}) & j = N.
\end{cases}
\]
The atom at $e^{ix}$ has the weight
\[
\frac{1}{\frac{3}{2} + \frac{1}{2} \sum_{j=1}^N a_j \cot^2 \frac{x - \theta_j}{2}}.
\]
Asymptotically,
\[
\lim_{k \rightarrow +\infty} e^{i x_{jk}} = e^{i \theta_j}, \lim_{k \rightarrow - \infty} e^{i x_{jk}} = e^{i \theta_{j+1}}
\]
($j$ modulo $N$).
\end{Cor}

\begin{Prop}
\label{Prop:Absolute-continuity}
Let $\mu \in \mc{L}$.
\begin{enumerate}
\item
$\mu$ is absolutely continuous, singular, or purely atomic if and only if $W(\mu)$ is.
\item
$\mu$ is absolutely continuous with a strictly positive density if and only if $W(\mu)$ is.
\item
The components of $\supp(\mu^{ac})$ are $2 \pi$-periodic. In particular there is either one or infinitely many components. The number of components of $\supp(W(\mu)^{ac})$ is either one less than or equal to the number of components of $\supp(\mu^{ac}) \cap [0, 2 \pi)$, depending on whether $0$ is in the interior of the support.
\end{enumerate}
\end{Prop}

\begin{proof}
Part (a) and one direction of part (b) follow directly from the definition of $W$. Next, suppose that $\mu$ has zero density at $x$. Then from \eqref{Density}, $\mu$ also has zero density at all $x + 2 \pi n$, and therefore $W(\mu)$ has zero density at $e^{-ix}$. Parts (b) and (c) follow.
\end{proof}

\begin{Prop}
\label{Prop:Continuity-W}
$W$ is weakly continuous; $W^{-1}$ is weakly continuous as a map to $\mc{L} \mod \delta_{2 \pi}$, but not as a map to $\mc{L}$.
\end{Prop}

\begin{proof}
Continuity of $W$ follows directly from relation \eqref{Wrapping}. If $W(\mu_n) \rightarrow W(\mu)$, it follows that $e^{i f_{\mu_n}(z)} \rightarrow e^{i f_\mu(z)}$ uniformly on compact subsets of $\mf{D}$. So $f_{\mu_n} - f_\mu \rightarrow 0 \mod 2 \pi$, but may not converge to zero. It remains to observe that weak convergence $\mod \delta_{2 \pi}$ is equivalent to uniform convergence $\mod 2 \pi$ of $F$-transforms on compact sets.
\end{proof}

\begin{Lemma}
\label{Lemma:Infinitesimal}
Let $\set{\nu_{nk} : n \in \mf{N}, 1 \leq k \leq k_n} \subset \mc{ID}^{\sutimes}_\ast$ be an infinitesimal triangular array, so that
\[
\lim_{n \rightarrow \infty} \max_{1 \leq k \leq k_n} \nu_{nk} \left( \set{ \zeta \in \mf{T} : \abs{\zeta - 1} \geq \eps} \right) = 0.
\]
Then there exists a triangular array $\set{\mu_{nk} : n \in \mf{N}, 1 \leq k \leq k_n} \subset \mc{L}$ such that $W(\mu_{nk}) = \nu_{nk}$ which is also infinitesimal, that is,
\[
\lim_{n \rightarrow \infty} \max_{1 \leq k \leq k_n} \mu_{nk} \left( \set{ x \in \mf{R} : \abs{x} \geq \eps} \right) = 0.
\]
\end{Lemma}

\begin{proof}
Fix small $\eps, \delta > 0$. Choose $n_0$ large enough that for all $n \geq n_0$ and all $k$,
\[
\nu_{nk} \left( \set{ \zeta \in \mf{T} : \abs{\zeta - 1} \geq \eps} \right) < \delta.
\]
Then for any $\mu_{nk} \in \mc{L}$ with $W(\mu_{nk}) = \nu_{nk}$,
\[
\mu_{nk}\left(\bigcap_{\ell \in \mf{Z}} \set{x \in \mf{R} : \abs{x - 2 \pi \ell} > \eps/2} \right) < \delta.
\]
On the other hand, using Poisson summation,
\[
\begin{split}
\mu_{nk} \left(\bigcap_{\ell \in \mf{Z}} \set{x \in \mf{R} : \abs{x - 2 \pi \ell} \leq \eps} \right)
& = \sum_{\ell \in \mf{Z}} \lim_{y \downarrow 0} \int_{-\eps}^\eps \Imm \frac{1}{F_{\mu_{nk}}(x + i y) + 2 \pi \ell} \,dx \\
& = \frac{1}{2\pi} \lim_{y \downarrow 0} \int_{-\eps}^\eps \frac{1 - e^{-2 \Imm F_{\mu_{nk}}(x + i y)}}{\abs{1 - e^{i F_{\mu_{nk}}(x + i y)}}^2} \,dx
\end{split}
\]
since the series consists of non-negative terms. Denote
\[
C = \frac{1}{2 \pi} \sup_{\substack{\abs{x} < \eps \\ 0 < y < \eps}} \frac{1 - e^{-2 \Imm F_{\mu_{nk}}(x + i y)}}{\abs{1 - e^{i F_{\mu_{nk}}(x + i y)}}^2}.
\]
If $C < \infty$, then we get $2 C \eps \geq 1 - \delta$, giving a contradiction for sufficiently small $\eps, \delta$. Thus $C = \infty$. So for some $x, y$ as above and some $\ell \in \mf{Z}$, $\abs{F_{\mu_{nk}}(x + i y) - 2 \pi \ell} < \delta$. Since $F_{\mu_{nk}}$ is continuous, this implies that for sufficiently small $\eps$ the same property holds for all $x, y$ as above. By replacing $\mu_{nk}$ with $\mu_{nk} \uplus \delta_{2 \pi \ell}$, we may assume that $\ell = 0$. Note that
\[
\frac{1}{2 \pi} \frac{1 - e^{-2 \Imm F_{\mu_{nk}}(x + i y)}}{\abs{1 - e^{i F_{\mu_{nk}}(x + i y)}}^2} - \frac{1}{\pi} \frac{\Imm F_{\mu_{nk}}(x + i y)}{\abs{F_{\mu_{nk}}(x + i y)}^2}
= \frac{1}{\pi} \frac{O(\abs{F_{\mu_{nk}}(x + i y)}^2)}{\abs{F_{\mu_{nk}}(x + i y)}^2} < C'
\]
for some constant $C'$ independent of $x,y$. Thus
\[
\mu_{nk} \left(\bigcap_{\ell \neq 0} \set{x \in \mf{R} : \abs{x - 2 \pi \ell} \leq \eps} \right)
< 2 C' \eps,
\]
and so
\[
\mu_{nk} \left( \set{ x \in \mf{R} : \abs{x} \geq \eps} \right) < \delta + 2 C' \eps.
\]
Since $C'$ decreases with $\eps$, by choosing $\eps$ sufficiently small we may achieve the bound of $2 \delta$, proving that the array is infinitesimal.
\end{proof}

\begin{Lemma}
\label{Lemma:Injective}
Let $\mu \in \mc{L}$. $F_\mu$ is injective if and only if $\eta_{W(\mu)}$ is.
\end{Lemma}

\begin{proof}
Suppose $F_\mu(z) = F_\mu(w)$. Then $\eta_{W(\mu)} (e^{i z}) = \eta_{W(\mu)} (e^{i w})$. If $z \neq w$, then by definition of $\mc{L}$, $2 \pi \nmid (z - w)$, and so $e^{i z} \neq e^{i w}$. Conversely, if $\eta_{W(\mu)} (e^{i z}) = \eta_{W(\mu)} (e^{i w})$, then $F_\mu(z) = F_\mu(w) + 2 \pi n = F_\mu(w + 2 \pi n)$ for some $n$. If $e^{i z} \neq e^{i w}$, then $z \neq w + 2 \pi n$.
\end{proof}

\section{Applications}
\label{Section:Applications}

\subsection{Multiplicative Belinschi-Nica transformations and the divisibility indicator}

Transformations $\mf{M}_t$, the multiplicative analogs of the Belinschi-Nica transformations, were defined in Section 4 of \cite{Ariz-Hasebe-Semigroups}. The definition requires some care since because of the non-uniqueness of multiplicative convolution powers one has different choices for $\mf{M}_t$. In \cite{Zhong-Free-BM} this ambiguity was resolved by chosing measures with positive mean.  Here, we define $\mf{M}_t$ on $\mc{ID}^{\sutimes}_\ast$ by
\begin{equation}
\label{Mt}
\mf{M}^{(n)}_t(W(\mu)) = W(\mf{B}_t(\mu)),
\end{equation}
where
\[
\mu \in \mc{L}_n := \set{\widetilde{\mu} \in \mc{L} : \Imm f_{\widetilde{\mu}}(0)\in[2n\pi,2(n+1)\pi)}.
\]
Note that for each $n$, $W : \mc{L}_n \rightarrow \mc{ID}^{\sutimes}_\ast$ is a bijection, so each $\mf{M}^{(n)}_t$ is well defined and onto. We will denote $\mf{M}_t = \mf{M}_t^{(0)}$. It is also easy to see, as in Lemma~\ref{Lemma:L-convolution}, that
\begin{equation}
\label{B_t-L_n}
\mf{B}_t(\mc{L}_n) \subset \mc{L}_n.
\end{equation}

The following is a noticeably shorter proof of Theorem 4.5  from \cite{Ariz-Hasebe-Semigroups}.

\begin{Cor}
$\set{\mf{M}^{(n)}_t : t \geq 0}$ form a semigroup: $\mf{M}^{(n)}_t \circ \mf{M}^{(n)}_s =\mf{M}^{(n)}_{t+s}$. Also, $\mf{M}^{(n)}_1 = \mf{M}$, the multiplicative Boolean-to-free Bercovici-Pata bijection.
\end{Cor}

\begin{proof}
For $\mu \in \mc{L}_n$, using \eqref{B_t-L_n} and Theorem 1.1 from \cite{Belinschi-Nica-B_t},
\[
\mf{M}^{(n)}_t \circ \mf{M}^{(n)}_s(W(\mu)) = W(\mf{B}_t \circ \mf{B}_s(\mu)) = W(\mf{B}_{t+s}(\mu)) = \mf{M}^{(n)}_{t+s}(W(\mu)).
\]
Similarly, using Theorem 1.2 from \cite{Belinschi-Nica-B_t} and Proposition~\ref{Prop:BP-intertwiner},
\[
\mf{M}^{(n)}_1(\nu_{\sutimes}^{\gamma, \sigma}) = \mf{M}^{(n)}_1(W(\nu_{\uplus}^{\alpha, \tau})) = W(\mf{B}_1(\nu_{\uplus}^{\alpha, \tau})) = W(\nu_{\boxplus}^{\alpha, \tau}) = \nu_\boxtimes^{\gamma, \sigma}
\]
where in the second identity we used the fact that $\mf{B}_1(\mu\boxplus \delta_{2\pi})=\mf{B}_1(\mu)\boxplus \delta_{2\pi}$ for $\mu\in\mc{L}$.
\end{proof}

In the same manner,  as pointed out in \cite{Ariz-Hasebe-Semigroups}, the multiplicative analog of the commutation relation proved in \cite{Belinschi-Nica-B_t} $$  \left(\mu^{\boxplus p} \right)^{\uplus q},
=  \left(\mu^{\uplus p'} \right)^{\boxplus q'}, ~~~q' = 1 - p + p q, p q = p' q'$$
needs some care because the multiplicative powers are multi-valued. We will solve this problem by the use of  Propostion \ref{Prop:ID-preimage}.
Define $$\nu^{\boxtimes_\phi p}=W(\mu^{\boxplus p}),  \quad \text{and} \quad \nu^{\sutimes_\phi p}=W(\mu^{\uplus p}) $$
where $\mu \in \mc{L}_{[\phi]}$ and $[\phi]$ denotes the integer part of $\phi$.
Now we can recover Proposition 4.4 in \cite{Ariz-Hasebe-Semigroups}.

\begin{Cor}\label{Cor:commutation}
Let $a=(2\pi)^{-1}\arg m_1(\nu)\in \mathbb{R}$ be an arbitrary argument of the first moment $m_1(\nu)$. For any $\nu \in \mc{ID}^{\sutimes}_\ast$, and $p \geq 1$ and $1 - 1/p < q$,
\[
\left(\nu^{\boxtimes_{a}~p} \right)^{\sutimes_{ap} ~q} = \left(\nu^{\sutimes_a~q'} \right)^{\boxtimes_{aq'} p'},
\]
where $q' = 1 - p + p q$, $p q = p' q'$.
\end{Cor}

\begin{proof}
From Proposition \ref{Prop:A-on-functions} and Theorem \ref{Thm:Wrapping} one can see that for any choice of argument of $m_1(\nu)$ there is a unique $\mu$ such that $\nu=W(\mu)$ and $\Imm f_\mu(0)=\arg(\eta_\nu'(0))=\arg m_1(\nu)=2\pi a$, so that $\mu \in \mc{L}_{[a]}$.  Since $\Imm f_{\mu^{\boxplus p}}(0)=2 \pi ap$ then $\mu^{\boxplus p}\in\mc{L}_{[ap]}$ and we have
\[
\left(\nu^{\boxtimes_{a}~p} \right)^{\sutimes_{ap} ~q}=\left(W(\mu^{\boxplus p}) \right)^{\sutimes_{ap}~q}
= W \left( \left(\mu^{\boxplus p} \right)^{\uplus q} \right)
\]
Similarly, since $\Imm f_{\mu^{\uplus q'}}(0)=2\pi a q'$ one sees, that
\[ \left(\nu^{\sutimes_a~q'} \right)^{\boxtimes_{aq'} p'}= W\left(\left(\mu^{\uplus p'} \right)^{\boxplus q'}\right).
\]

Using Proposition 3.1 from \cite{Belinschi-Nica-B_t} we obtain the claim.
 \end{proof}

Recall that the (additive) divisibility indicator was defined in Definition~1.4 of \cite{Belinschi-Nica-B_t} as
\[
\sup \set{t \geq 0: \mu \in \mf{B}_t(\mc{P}(\mf{R}))},
\]
while the multiplicative divisibility indicator was defined in Definition~4.6 of \cite{Ariz-Hasebe-Semigroups} as
\[
\theta(\nu):=\sup \set{t \geq 0 : \nu \in \mf{M}_t(\mc{ID}^{\sutimes}_\ast)}.
\]

Note that in the second definition, we have excluded the Haar measure and that we omitted the superscript $n$ in $\mf{M}_t$. That we can choose any $n$ is a consequence of the next proposition.

\begin{Prop}
For $\mu \in \mc{L}$, the additive divisibility indicator of $\mu$ is equal to the multiplicative divisibility indicator of $W(\mu)$.
\end{Prop}

\begin{proof}
Let $\mu \in \mf{B}_t(\mc{P}) \cap \mc{L}_n$. Then by Lemma~\ref{Lemma:L-convolution} and \eqref{B_t-L_n}, in fact $\mu \in \mf{B}_t(\mc{L}_n)$. Denote $\widetilde{\mu} = \mu \uplus \delta_{- 2 \pi n}$. Then by Proposition 3.8 in \cite{Ariz-Hasebe-Semigroups}, $\widetilde{\mu} \in \mf{B}_t(\mc{L}_0)$. So \eqref{Mt} implies that $W(\widetilde{\mu}) \in \mf{M}_t(\mc{ID}^{\sutimes}_\ast)$. Conversely, let $\mu \in \mc{L}_n$ with $W(\mu)$ in the image of $\mf{M}_t$. Since $W : \mc{L}_0 \rightarrow \mc{ID}^{\sutimes}_\ast$ is onto, for some $\widetilde{\mu} \in \mc{L}$,
\[
W(\mu) = \mf{M}_t(W(\widetilde{\mu}) = W(\mf{B}_t(\widetilde{\mu})).
\]
Therefore
\[
\mu \uplus \delta_{2 \pi n} = \mf{B}_t(\widetilde{\mu})
\]
for some $n$, which again implies that $\mu$ is in the image of $\mf{B}_t$.
\end{proof}

The following result is the analog of Proposition~5.1 from \cite{Belinschi-Nica-B_t}, and follows immediately from it and elementary properties of $W$.

\begin{Prop}
Let $\nu \in \mc{ID}^{\sutimes}_\ast$ and $t > 0$. Then
\begin{enumerate}
\item
$\mf{M}_t(\nu)$ has no singular continuous part.
\item
$\mf{M}_t(\nu)$ has at most $[1/t]$ atoms for $t \leq 1$ and at most one atom for $t > 1$.
\item
The absolutely continuous part of $\mf{M}_t(\nu)$ is zero if and only if $\nu$ is a point mass. Its density is analytic whenever positive and finite.
\end{enumerate}
\end{Prop}

Similarly, the following properties of the multiplicative divisibility indicator proved in \cite[Theorem 4.8]{Ariz-Hasebe-Semigroups} are now a direct consequence from their additive counterparts  proved in \cite{Belinschi-Nica-B_t}, parts (a-c), and \cite{Ariz-Hasebe-Semigroups}, parts (d),(e).

\begin{Cor}
 We consider a probability measure $\mu \in \mc{ID}^{\sutimes}_\ast $.
\begin{enumerate}
\item
 $\mu^{\boxtimes t}$ exists for any  $t \geq \max\{1-\theta(\mu),0 \}$.
\item
 $\mu$ is $\boxtimes$-infinitely divisible if and only if $\theta(\mu) \geq 1$.
\item
 $\theta(\mathbb{M}_t(\mu)) = \theta(\mu)+t$ for any  $t \geq -\theta(\mu)$.
\item
 $\theta(\mu^{\sutimes t}) = \frac{1}{t}\theta(\mu)$ for any $t > 0$.
\item
 $\theta(\mu^{\boxtimes \, t})-1 = \frac{1}{t}(\theta(\mu)-1)$ for any $t > \max \{1- \theta(\mu), 0 \}$.
\end{enumerate}
\end{Cor}

\begin{Example}
\label{Example:Atoms}
Let $\sigma$ be a measure on $\mf{T}$ with finite support. Then the measure in $\nu \in\mc{L}$ corresponding to the pair $(\beta, \sigma)$ in Lemma \ref{Cara} is purely atomic and thus $\theta(\nu)=0.$  Since $\mf{M}_1 = \mf{M}$, it follows that $\theta(\mf{M}(\nu))=1$. Thus any measure whose $\Sigma$ transform is of the form \eqref{Sigma-transform} with $\sigma$ of finite support has the multiplicative divisibility indicator $1$. This includes the multiplicative free Gaussian, with the $\Sigma$ transform
\[
\Sigma(z) = \exp \left( \frac{1 + z}{1 - z} \right),
\]
and the multiplicative free Poisson law, with
\[
\Sigma(z) = \exp \left( \frac{1 + \zeta z}{1 - \zeta z} \right).
\]
\end{Example}

The following result was proved in the algebraic multivariate setting in Proposition 1.10 in \cite{Nica-Subordination}. It is a re-formulation of Theorem 2.8 in \cite{Zhong-Free-BM}, and generalizes Theorem 1.6 in \cite{Belinschi-Nica-B_t}. See also \cite{Ans-Free-evolution-2}.
\begin{Thm}
\label{Thm:Additive-evolution}
Given $\tau\in \mc{ID}^\boxplus$ and $\nu \in \mc{P}(\mf{R})$, denote $\mu = \tau \boxright \nu$. Then
\[
\mf{B}_t(\tau \boxright \nu) = \tau \boxright (\nu \boxplus \tau^{\boxplus t}),
\]
\end{Thm}

\begin{proof}
The assumption says that
\begin{equation}\label{z1} \phi_\tau(F_\nu(z))=z-F_\mu(z), \qquad z\in \mf{C}^+,
\end{equation}
and the desired conclusion is
\begin{equation}\label{z2}\phi_\tau(F_{\nu\boxplus \tau^{\boxplus t}}(z))=z-F_{\mathbb{B}_t(\mu)}(z),  \qquad z\in \mf{C}^+.
\end{equation}
These are related by Theorem 2.8 in \cite{Zhong-Free-BM}.
\end{proof}

In the remaining part of the section we will omit the superscripts on $\mf{M}_t$ and subscripts on the free and boolean powers to avoid excessive notation.  The choice of these is done exactly as in Corollary \ref{Cor:commutation}.

The following result generalizes the main theorem (Theorem 1.1) in \cite{Zhong-Free-BM} from multiplicative free Gaussian to general $\tau$.

\begin{Thm}
Given a pair of probability measures $\nu, \tau \in \mc{ID}^{\sutimes}_\ast$, we have
\[
\mf{M}_t(\tau \msubord \nu) = \tau \msubord (\nu \boxtimes \tau^{\boxtimes t}).
\]
\end{Thm}

\begin{proof}
Denote $\mu = \tau \msubord \nu$. This means that
\begin{equation}\label{z3}\Sigma_\tau(\eta_\nu(z))=\frac{z}{\eta_\mu(z)}, z\in \mf{D}
\end{equation}

Let $\tau=W(\tilde{\tau})$, $\mu=W(\tilde{\mu})$ and $\nu=W(\tilde{\nu})$. Then evaluating \eqref{z3} at  $e^{iz}$ we get
$$\Sigma_{W(\tilde{\tau})}(\eta_{W(\tilde{\nu})}(e^{iz}))=\frac{e^{iz}}{ \eta_{W(\tilde\mu)}(e^{iz})}.$$
By definition, for all $\sigma$, $\exp(i \varphi_{\sigma}(z)) = \Sigma_{W(\sigma)}(e^{iz})$ and $\eta_{W(\sigma)}(e^{iz})=\exp(i F_\sigma(z))$, Thus
$$\exp(i \varphi_{\tilde{\tau}}(F_\mu(z))=\frac{e^{iz}}{\exp(iF_{\tilde \mu}(z))}$$ which is equivalent to $$\phi_{\tilde\tau}(F_{\tilde\nu}(z))=z-F_{\tilde\mu}(z) + 2 \pi k$$
for some $k$. By adjusting the pre-image of $\mu$, we may assume that $k=0$.
By Theorem~\ref{Thm:Additive-evolution} we get
\[\phi_{\tilde{\tau}}(F_{\tilde{\nu}\boxplus \tilde{\tau}^{\boxplus t}}(z))=z-F_{\mathbb{B}_t(\tilde{\mu})}(z),
\]
which by reverting the steps above implies that \[\Sigma_{W(\tilde{\tau})}(\eta_{W(\tilde{\nu}\boxplus \tilde{\tau}^{\boxplus t})}(e^{iz}))=\frac{e^{iz}}{ \eta_{W(\mathbb{B}_t(\tilde{\mu}))}(e^{iz})}.\]

Finally, since $\tau=W(\tilde{\tau})$,  $W(\tilde{\nu}\boxplus \tilde{\tau}^{\boxplus t})=\nu\boxtimes \tau^{\boxtimes t}$ and $W(\mathbb{B}_t(\tilde{\mu}))=\mathbb{M}_t(\mu)$, we obtain
\begin{equation}\label{z4}\Sigma_\tau(\eta_{\nu\boxtimes \tau^{\boxtimes t} }(z))=\frac{z}{\eta_{\mathbb{M}_t(\mu)}(z)}, z\in \mf{D}.
\end{equation}
The result follows.
\end{proof}

The following are analogs of Proposition~5.3 from \cite{Nica-Subordination}, and follow from the preceding theorem.

\begin{Cor}
Let $\tau \in \mc{ID}^{\sutimes}_\ast$. Then $\tau \msubord \tau = \mf{M}(\tau)$, and more generally, for $s, t \geq 0$,
\[
(\tau^{\boxtimes s}) \msubord (\tau^{\boxtimes t}) = \left( \mf{M}_t(\tau) \right)^{\boxtimes s}.
\]
\end{Cor}

By similar methods, we also easily obtain the following analogs of Corollary 4.13 in \cite{Nica-Subordination} and Lemma~7 in \cite{Ans-Free-evolution-2}.

\begin{Prop}
$\mu \msubord \nu \in \mc{ID}^{\boxtimes}$ whenever $\mu \in \mc{ID}^\boxtimes$ or $\nu = \mu \boxtimes \nu'$.
\end{Prop}

\begin{Thm}
Let $\mu \in \mc{ID}^{\sutimes}_\ast$, and let $ f:(0,\infty)\times \mathbb{D}\to\mathbb{C}$ be defined by
  \begin{equation}f(t,z)=\frac{\eta_{\mathbb{M}_t(\mu)}(z)}{z},  \quad \forall t>0, \forall z\in\mathbb{D}.
\end{equation}
Then f satisfies the following multiplicative version of the inviscid Burgers' equation
\begin{equation}
\label{Log-evolution}
\frac{\partial f}{\partial t}(t,z)=z\log f(t,z)\frac{\partial f}{\partial z}(t,z),\quad    t>0,z\in\mathbb{D}.
\end{equation}
\end{Thm}

\begin{proof}
Again let  $\mu=W(\tilde\mu)$. According to Theorem 1.5 in \cite{Belinschi-Nica-B_t},
\begin{equation}h(t,z)=F_{\mathbb{B}(t)(\mu)}(z)-z,  \quad \forall t>0, \forall z\in\mathbb{D}.
\end{equation}
satisfies
\begin{equation}\frac{\partial h}{\partial t}(t,z)= h(t,z)\frac{\partial h}{\partial z}(t,z),\quad    t>0,z\in\mathbb{D}.
\end{equation}

Since $W(\mathbb{B}_t(\tilde\mu))=\mathbb{M}_t(W(\tilde\mu))=\mathbb{M}_t(\mu)$ then
\[\exp(i(F_{\mathbb{B}(t)(\mu)}(z)-z))=\frac{\eta_{\mathbb{M}_t(\mu)}(e^{iz})}{e^{iz}}
\]
which is means that
\[
h(t,z)=-i\log f(t,e^{iz}) \qedhere
\]
\end{proof}

\begin{Remark}
Let
\[
\eta_{\mf{M}_t(\mu)}(z) = z \exp \int_{\mf{T}} \frac{\zeta + z}{\zeta - z} \,d\mu_t(\zeta),
\]
so that $\mu_t$ is the (reflected) Boolean L{\'e}vy measure of $\mf{M}_t(\mu)$. Then according to Section~4.2 of \cite{Lawler-Conformally-invariant-processes}, the equation \eqref{Log-evolution} is precisely the radial Loewner's equation driven by the family $\set{\mu_t}$.
\end{Remark}

\begin{Remark}
In \cite{Ariz-Hasebe-Subordination}, Arizmendi and Hasebe defined (in our notation)
\[
\mf{B}_\sigma (\nu) = \nu \msubord \sigma, \quad \sigma, \nu \in \mc{P}(\mf{T}),
\]
and showed that they share many properties of $\mf{B}_t$, such as $\mathbb{B}_\sigma(\nu_1\boxtimes\nu_2)=\mathbb{B}_\sigma(\nu_1)\boxtimes\mathbb{B}_\sigma(\nu_2)$ and $\mathbb{B}_{\sigma_1}\circ\mathbb{B}_{\sigma_2}=\mathbb{B}_{\sigma_2\calt\sigma_1}$. They also defined $\mf{A}_\sigma(\mu) = \mu \boxright \sigma$ for $\sigma, \mu \in \mc{P}(\mf{R})$, and proved similar properties for these transformations. Using the techniques above, many properties of $\mf{B}_{W(\sigma)}$ can now be derived from those of $\mf{A}_\sigma$.
\end{Remark}

\subsection{Further results}

\subsubsection{Free convolution powers}
Recall that in general, free convolution powers are defined only for $\nu \in \mc{ID}^{\sutimes}_\ast$. So the following results are well suited for our methods. They were proved in (a) Theorem~3.5 (d) Proposition 5.3 in \cite{Bel-Ber-Partially-mult} (c) Theorem 3.2 of \cite{Zhong-regularity-multiplicative} (e) Theorems 3.8, 3.11 of \cite{Huang-Zhong-supports-multiplicative}. A version of part (b) appears in Proposition 5.3 of \cite{Bel-Ber-Partially-mult}, the formulation below appears to be new.

\begin{Prop}
Let $\nu \in \mc{ID}^{\sutimes}_\ast$ and $t \geq 1$. Denote $\nu_t = \nu^{\boxtimes t}$.
\begin{enumerate}
\item
$\nu_t$ is well defined up to a rotation by $e^{2 \pi i n t}$ ($n \in \mf{Z}$) and is in $\mc{ID}^{\sutimes}_\ast$.
\item
$\zeta$ is an atom of $\nu_t$ if and only if for some $\alpha \in \mf{R}$, $e^{-i t \alpha} = \zeta$ and $e^{-i \alpha}$ is an atom of $\nu$, with $\nu(\set{e^{-i \alpha}}) > 1 - 1/t$, in which case
\[
\nu_t(\set{\zeta}) = t \nu(\set{e^{-i \alpha}} - (t-1).
\]
\item
If $\alpha, \beta$ are atoms of $\nu_t$, then $\nu_t(I) > 0$, where $I$ is an open arc with endpoints $\alpha, \beta$.
\item
$\nu_t$ has no singular continuous part, and the density of its absolutely continuous part is analytic wherever it is positive.
\item
The number of components of $\supp(\nu_t^{ac})$, and of $\supp(\nu_t)$, is a decreasing function of $t$.
\end{enumerate}
\end{Prop}

\begin{proof}
By Proposition~\ref{Prop:ID-preimage}, we may choose $\mu \in \mc{L}$ so that $\nu = W(\mu)$ and $\nu_t = W(\mu^{\boxplus t})$. We will prove part (b). The remaining parts follow by similar methods, from (a) Theorem~2.5 (c) Proposition~3.3 (d) Theorem~3.4 in \cite{Belinschi-Bercovici-Partially-defined} or Theorem~5.1 in \cite{Bel-Ber-Partially-mult} (e) Theorem 3.8 in \cite{Huang-supports-additive}. For part (b), by Theorem 3.1 in \cite{Belinschi-Bercovici-Partially-defined},
\[
\mu^{\boxplus t}(\set{t \alpha}) = \max \left(0, t \mu(\set{\alpha}) - (t-1)\right),
\]
and the result follows from Proposition~\ref{Prop:Atoms0}.
\end{proof}

\begin{Cor}
\label{Cor:L-atoms-2}
Let $\mu \in \mc{L}$ and $\alpha, \beta$ be atoms of $\mu$. Then
\[
\min( \mu(\set{\alpha}), \mu(\set{\beta})) < \frac{2 \pi}{\abs{\alpha - \beta}}.
\]
\end{Cor}

\begin{proof}
Let $2 \pi n < \abs{\alpha - \beta} \leq 2 \pi (n+1)$, and let
\[
t = \frac{2 \pi (n+1)}{\abs{\alpha - \beta}}.
\]
Then $1 \leq t < 1 + \frac{2 \pi}{\abs{\alpha - \beta}}$. Suppose both $\mu(\set{\alpha}), \mu(\set{\beta}) \geq \frac{2 \pi}{\abs{\alpha - \beta}} > t - 1$. Then $\mu^{\boxplus t} \in \mc{L}$ has atoms at $t \alpha$ and $t \beta$, with $t \abs{\alpha - \beta} = 2 \pi (n+1)$, contradicting Corollary~\ref{Cor:L-atoms-1}.
\end{proof}

\subsubsection{Limit theorems}

The results in this section are proven only for measures in $\mc{ID}^{\sutimes}_\ast$, and so are typically weaker than known ones. On the other hand, the proofs are much simpler and shorter.

The following proposition is the restriction to $\mc{ID}^{\sutimes}_\ast$ of Theorem~4.3 in \cite{Bercovici-Wang-Multiplicative} and Theorem~3.5, 4.1 in \cite{Wang-Boolean}.

\begin{Prop}
\label{Prop:Limit}
Suppose $\set{\nu_{nk} : n \in \mf{N}, 1 \leq k \leq k_n} \subset \mc{ID}^{\sutimes}_\ast$ form an infinitesimal triangular array. The following are equivalent.
\begin{enumerate}
\item
The sequence $\nu_{n 1} \utimes \nu_{n 2} \utimes \cdots \utimes \nu_{n k_n}$ converges weakly to $\nu_{\sutimes}^{\gamma,\sigma}$.
\item
The sequence $\nu_{n 1} \boxtimes \nu_{n 2} \boxtimes \cdots \boxtimes \nu_{n k_n}$ converges weakly to $\nu_{\boxtimes}^{\gamma,\sigma}$.
\item
The sequence $\nu_{n 1} \circledast \nu_{n 2} \circledast \cdots \circledast \nu_{n k_n}$ converges weakly to $\nu_{\circledast}^{\gamma,\sigma}$.
\end{enumerate}
\end{Prop}

\begin{proof}
Suppose $\nu_{n 1} \putimes \nu_{n 2} \putimes \cdots \putimes \nu_{n k_n}$ converges weakly to $\nu_{\sutimes}^{\gamma,\sigma}$. Using Lemma~\ref{Lemma:Infinitesimal}, we may choose an infinitesimal array $\set{\mu_{nk}} \subset \mc{L}$ so that $\nu_{n k} = W(\mu_{n k})$. Then by Proposition~\ref{Prop:Continuity-W},
\[
(\mu_{n 1} \mod \delta_{2 \pi}) \uplus (\mu_{n 2} \mod \delta_{2 \pi}) \uplus \ldots \uplus (\mu_{n k_n} \mod \delta_{2 \pi}) \rightarrow \nu_\uplus^{\alpha, \tau} \mod \delta_{2 \pi},
\]
where $\tau$ and $\sigma$ are related by equation \eqref{Sigma-Tau}. Therefore
\[
\mu_{n 1} \uplus \mu_{n 2} \uplus \ldots \uplus \mu_{n k_n} \rightarrow \nu_\uplus^{\alpha, \tau} \uplus \delta_{2 \pi \ell_n}
\]
for some integers $\ell_n$. By the additive Bercovici-Pata bijections \cite{BerPatDomains,Chistyakov-Gotze-Limit-I,Bercovici-Wang-Additive,Wang-Boolean}, it then follows that
\[
\mu_{n 1} \boxplus \mu_{n 2} \boxplus \ldots \boxplus \mu_{n k_n} \rightarrow \mu_\boxplus^{\alpha, \tau} \boxplus \delta_{2 \pi \ell_n}.
\]
So applying $W$ again,
\[
\nu_{n 1} \boxtimes \nu_{n 2} \boxtimes \cdots \boxtimes \nu_{n k_n} \rightarrow \nu_{\boxtimes}^{\beta, \sigma}.
\]
The other implications are similar.
\end{proof}

The following proposition is the restriction to $\mc{ID}^{\sutimes}_\ast$ of Proposition~5.6 in \cite{Ans-Williams-Chernoff}.

\begin{Prop}
Suppose $\set{\nu_n}_{n=1}^\infty \subset \mc{ID}^{\sutimes}_\ast$ and the sequence $\underbrace{\nu_{n} \utimes \nu_{n} \utimes \cdots \utimes \nu_{n}}_{k_{n}}$ converges weakly to $\nu_{\sutimes}^{\gamma,\sigma}$. Then for any $\beta \in \mf{R}$ with $\gamma = e^{i \beta}$, there exist $\lambda_n \in \mf{T}$, $\lambda_n^{k_n} = 1$ such that for $\tilde{\nu}_n = \delta_{\lambda_n} \circlearrowright \nu_{n}$,
the sequence $\underbrace{\tilde{\nu}_{n} \circlearrowright \tilde{\nu}_{n} \circlearrowright \cdots \circlearrowright \tilde{\nu}_{n}}_{k_{n}}$ converges weakly to $\nu_{\circlearrowright}^{\beta,\sigma}$.
\end{Prop}

\begin{proof}
Let $\nu_n = W(\mu_n)$. Then by Proposition~\ref{Prop:Continuity-W},
\[
(\mu_n \mod \delta_{2 \pi})^{\uplus k_n} \rightarrow \nu_\uplus^{\alpha, \tau} \mod \delta_{2 \pi},
\]
where $(\alpha, \tau)$ and $(\gamma, \sigma)$ are related by equation \eqref{Sigma-Tau}. Therefore
\[
(\delta_{2 \pi \ell_n / k_n} \uplus \mu_n)^{\uplus k_n} \rightarrow \nu_\uplus^{\alpha, \tau}
\]
for some integers $\ell_n$. Denote $W(\nu_{\rhd}^{\alpha, \tau}) = \nu^{\widetilde{\beta}, \sigma}_\circlearrowright$, and let $b_n = (\beta - \widetilde{\beta})/(2 \pi k_n)$. Then
\[
(\delta_{2 \pi (\ell_n + b_n) / k_n} \uplus \mu_n)^{\uplus k_n} \rightarrow \nu_\uplus^{\alpha + (\beta - \widetilde{\beta}), \tau}
\]
By Theorem~1.3 in \cite{Ans-Williams-Chernoff}, it then follows that
\[
(\delta_{2 \pi (\ell_n + b_n) / k_n} \rhd \mu_n)^{\rhd k_n} \rightarrow \nu_{\rhd}^{\alpha  + (\beta - \widetilde{\beta}), \tau}.
\]
So applying $W$ again,
\[
(\delta_{\exp(2 \pi i \ell_n / k_n)} \circlearrowleft \nu_n)^{\circlearrowleft k_n} \rightarrow \nu_{\circlearrowleft}^{\beta, \sigma}.
\]
Finally, let $\lambda_n = \exp(2 \pi i (\ell_n + b_n) / k_n)$.
\end{proof}

The proof of the converse direction in Theorem~5.7 of that paper also follows easily, but was already short to begin with.

\begin{Remark}
It follows from the results in \cite{Goryainov-Hydrodynamic,Bauer-Chordal-Loewner-families,Franz-Hasebe-private} that any $\mu$ with an injective $F$-transform and finite variance arises as a weak limit
\[
\mu = \lim_{n \rightarrow \infty} \mu_{n 1} \rhd \mu_{n 2} \rhd \ldots \rhd \mu_{n k_n}
\]
for an infinitesimal triangular array $\set{\mu_{nk} : n \in \mf{N}, 1 \leq k \leq k_n} \subset \mc{P}(\mf{R})$. In particular, unlike in the setting of Proposition~\ref{Prop:Limit}, the limit need not be $\rhd$-infinitely divisible. On the other hand, it follows from the results in \cite{Rosenblum-Rovnyak,Franz-Hasebe-private}  that any $\nu$ with an injective $\eta$-transform arises as a weak limit
\[
\nu = \lim_{n \rightarrow \infty} \nu_{n 1} \circlearrowright \nu_{n 2} \circlearrowright \ldots \circlearrowright \nu_{n k_n} \rightarrow \nu,
\]for an infinitesimal triangular array $\set{\nu_{nk} : n \in \mf{N}, 1 \leq k \leq k_n} \subset \mc{P}(\mf{T})$. Combining this result with Lemmas~\ref{Lemma:Infinitesimal} and \ref{Lemma:Injective}, it follows that any $\mu \in \mc{L}$ with an injective $F$-transform is a limit of a monotone infinitesimal triangular array. This suggests that the restriction of finite variance above can be removed in general.
\end{Remark}

\begin{Remark}
Consider a weakly continuous family
\[
\set{\mu_t : t \geq 0} \cup \set{\mu_{s,t} : 0 < s < t}
\]
such that $\mu_t \rhd \mu_s = \mu_{t+s}$ (so that they form a monotone convolution semigroup) and $\mu_s \boxplus \mu_{s,t} = \mu_t$. Then in the terminology of \cite{BiaProcesses}, these measures form the distribution of a free additive L{\'e}vy process of the second kind (FAL2), which has freely independent but not stationary increments, and stationary Markov transition functions. Similarly, a weakly continuous family
\[
\set{\nu_t : t \geq 0} \cup \set{\nu_{s,t} : 0 < s < t}
\]
such that $\nu_t \circlearrowright \nu_s = \mu_{t+s}$ and $\nu_s \boxplus \nu_{s,t} = \nu_t$, form the distribution of a free unitary multiplicative L{\'e}vy process of the second kind (FUL2). Biane gave examples of FAL2 processes, and conjectured that non-trivial FUL2 processes do not exist. Wang in \cite{Wang-Monotone-CLT} showed that FAL2 processes with zero mean and finite variance do not exist.
\end{Remark}

\begin{Prop}
FUL2 processes exist if and only if there exist FAL2 processes with distributions $\set{\mu_t} \subset \mc{L}$.
\end{Prop}

\begin{proof}
Let $\set{\mu_t : t \geq 0} \cup \set{\mu_{s,t} : 0 < s < t}$ be a distribution of a FAL2 process. Then the same argument as in the proof of Lemma~\ref{Lemma:L-convolution} shows that all $\mu_{s,t} \in \mc{L}$ as well. Therefore $\set{W(\mu_t) : t \geq 0} \cup \set{W(\mu_{s,t}) : 0 < s < t}$ is a distribution of a FUL2 process.

Conversely, let $\set{\nu_t : t \geq 0} \cup \set{\nu_{s,t} : 0 < s < t}$ be a distribution of a FUL2 process. Then $\set{\mu_t} \subset \mc{ID}^{\circlearrowright}_\ast \subset \mc{ID}^{\sutimes}_\ast$, and by Proposition~\ref{Prop:ID-preimage}, we may choose a $\rhd$-semigroup $\set{\mu_t : t \geq 0}$ with $W(\mu_t) = \nu_t$. Also, weak continuity of the family implies by a standard argument that all $\mu_{s,t} \in \mc{ID}^{\boxtimes}_\ast \subset \mc{ID}^{\sutimes}_\ast$. Therefore we may choose $\mu_{s,t}$ so that $\mu_s \boxplus \mu_{s,t} = \mu_t$ and $W(\mu_{s,t}) = \nu_{s,t}$. We thus obtain a distribution of a FAL2 process.
\end{proof}

Finally, we note the relation via $W$ between the superconvergence results in additive limit theorems of \cite{Ber-Wang-Zhong-Superconvergence} and their multiplicative counterparts in \cite{Ans-Wang-Zhong}.

% as well as the relation between the additive and multiplicative results in \cite{Williams-Khinthchine-decomposition}. {\color{red} If $\nu \in \mc{ID}^{\sutimes}_\ast$, does the same hold for its indecomposable factors?}

\subsubsection{Free convolution}

Again, the known or desired results in this section are for measures in $\mc{P}(\mf{T})$, while our approach only works for measures in $\mc{ID}^{\sutimes}_\ast$. Nevertheless, some of the results are new.

Part (a) of the following proposition was proved in Theorem 3.1 of \cite{Belinschi-Atoms-mult}, part (d) in Theorem 4.1 of \cite{Bercovici-Wang-Indecomposable}. Parts (b,c) appear to be new.

\begin{Prop}
Let $\mu, \nu \in \mc{ID}^{\sutimes}_{\ast}$, neither a point mass. Then
\begin{enumerate}
\item
$\zeta \in \mf{T}$ is an atom of $\mu \boxtimes \nu$ if and only if there exist $\alpha, \beta \in \mf{T}$ such that $\zeta = \alpha \beta$ and $\mu(\set{\alpha}) + \nu(\set{\beta}) > 1$. In this case, $(\mu \boxtimes \nu)(\set{\zeta}) = \mu(\set{\alpha}) + \nu(\set{\beta}) - 1$.
\item
The absolutely continuous part of $\mu \boxtimes \nu$ is non-zero, and its density is analytic whenever it is positive and finite.
\item
The singular continuous part of $\mu \boxtimes \nu$ is zero.
\item
If $\zeta, \xi$ are atoms of $\mu \boxtimes \nu$, then $(\mu \boxtimes \nu)(I) > 0$, where $I$ is an open arc with endpoints $\zeta, \xi$.
\end{enumerate}
\end{Prop}

\begin{proof}
The results follow by the same methods from (a) \cite{BV96} (b,c) Theorem 4.1 in \cite{Belinschi-Lebesgue} (d) Theorem 2.5 in \cite{Bercovici-Wang-Indecomposable}.
\end{proof}

In \cite{Huang-Wang-Levy}, Huang and Wang proved the following result.

\begin{Thm*}
$\mu \in \mc{P}(\mf{R})$ has the property that for any $\nu \in \mc{P}(\mf{R})$, $\mu \boxplus \nu$ is absolutely continuous with a strictly positive density, if and only if $\mu$ itself is absolutely continuous with a strictly positive density and has an infinite second moment.
\end{Thm*}

They also proved a similar result for multiplicative convolution, which however did not require any moment conditions. The proof of the following proposition, a particular case of their result, explains this absence.

\begin{Prop}
$\mu \in \mc{ID}^\boxtimes_\ast$ has the property that for any $\nu \in \mc{ID}^{\sutimes}_\ast$, $\mu \boxtimes \nu$ is absolutely continuous with a strictly positive density, if and only if $\mu$ itself is absolutely continuous with a strictly positive density.
\end{Prop}

\begin{proof}
One direction is clear by taking $\nu = \delta_1$. Now let $\nu \in \mc{ID}^{\sutimes}_\ast$, so that $\nu = W(\widetilde{\nu})$ for some $\widetilde{\nu} \in \mc{L}$. Let $\mu \in \mc{ID}^\boxtimes_\ast$ be absolutely continuous with a strictly positive density. Then by Lemma~\ref{Lemma:Mult-ID}, $\mu \in \mc{ID}^{\sutimes}_\ast$, so $\mu = W(\widetilde{\mu})$, with $\widetilde{\mu} \in \mc{L}$. In particular, it has infinite second moment. By Proposition~\ref{Prop:Absolute-continuity}, $\widetilde{\mu}$ is absolutely continuous with a strictly positive density. By Proposition~\ref{Prop:ID-preimage}, it is in $\mc{ID}^\boxplus$. Therefore by the results of Huang and Wang for the real line, $\widetilde{\mu} \boxplus \widetilde{\nu}$  is absolutely continuous with a strictly positive density. Applying Proposition~\ref{Prop:ID-preimage} again, the conclusion follows.
\end{proof}

The following is a multiplicative version of Theorem 3.5 from Hasebe \cite{Hasebe-Monotone-ID}, and follows from it by applying $W$. Analogs of several other properties from that paper can be derived similarly.

\begin{Prop}If a $\circlearrowright$-infinitely divisible distribution $\nu$ contains an isolated atom at $\zeta$, then $\nu$ is of
the form $\nu = \nu(\{\zeta\})\delta_a + \nu_{ac}$, where $\nu_{ac}$ is absolutely continuous w.r.t. the Haar measure and $\zeta \notin supp \nu_{ac}$.
\end{Prop}

\begin{Remark}
Other identities easily obtained by the same methods include
\[
\mu = \mu^{\boxtimes t} \circlearrowleft \mu^{\sutimes (1-t)}
\]
using $\mu = \mu^{\boxplus t} \rhd \mu^{\uplus (1-t)}$ implicit in \cite{Accardi-Lenczewski-Salapata} for integer $n$ and stated in  \cite{Belinschi-Nica-B_t}  in terms of subordination,
\[
\mu \boxtimes \nu = (\mu \msubord \nu) \utimes (\nu \msubord \mu)
\]
which follows from its additive version in \cite[Proposition~1.11]{Nica-Subordination}, \cite[Theorem~4.1]{Bel-Ber-Subordination} and
\[
\Sigma_{\mu \msubord \nu} = \Sigma_\mu \circ \eta_\nu.
\]
using $\phi_{\mu \boxright \nu} = \phi_\mu \circ F_\nu$.
\end{Remark}

%{\color{red} Here there are two options for the multiplicative case,  either we have and differential equation (see Section 3.1 for the atom or there are no atoms.}

%{\color{red}  from Hasebe and Thorbjonsen }

%\begin{Thm}Any freely selfdecomposable probability law is unimodal
%\end{Thm}

%{\color{red} Here there is no other example than Cauchy which is known to be selfdecomposable on the class L , so the theorem above may be trivial, however one should check examples in Hasebe and Sakuma, http://arxiv.org/pdf/1508.01285v1.pdf, specially examples 4.10 and 4.9} New results?

%\bibliographystyle{amsalpha}
%\bibliography{bibdata}

\def\cprime{$'$} \def\cprime{$'$}
\providecommand{\bysame}{\leavevmode\hbox to3em{\hrulefill}\thinspace}
\providecommand{\MR}{\relax\ifhmode\unskip\space\fi MR }
% \MRhref is called by the amsart/book/proc definition of \MR.
\providecommand{\MRhref}[2]{%
  \href{http://www.ams.org/mathscinet-getitem?mr=#1}{#2}
}
\providecommand{\href}[2]{#2}

\end{document}